\documentclass{amsart}
\usepackage{amssymb}

\calclayout

\newtheorem{theorem}{Theorem}[section]
\newtheorem{conj}[theorem]{Conjecture}

\newtheorem{cor}[theorem]{Corollary}
\newtheorem{lem}[theorem]{Lemma}

\theoremstyle{definition}
\newtheorem{beisp}[theorem]{Example}
\newtheorem{definit}[theorem]{Definition}

\newtheorem*{rem*}{Remark}

\DeclareMathOperator{\out}{Out}

\DeclareMathOperator{\Inn}{Inn}

\DeclareMathOperator{\aut}{Aut}

\DeclareMathOperator{\id}{id}

\DeclareMathOperator{\So}{SO}
\DeclareMathOperator{\Gl}{GL}
\DeclareMathOperator{\Sl}{SL}

\DeclareMathOperator{\oo}{O}
\DeclareMathOperator{\Go}{GO}

\DeclareMathOperator{\Span}{Span}
\DeclareMathOperator{\M}{M}

\DeclareMathOperator{\chr}{Char}

\renewcommand{\epsilon}{\varepsilon}
\renewcommand{\phi}{\varphi}
\renewcommand{\rho}{\varrho}
\renewcommand{\theta}{\vartheta}

\renewcommand{\le}{\leqslant}

\renewcommand{\ge}{\geqslant}

\title[Isomorphy Classes of $k$-involutions of $\So(n, k,\beta)$, $n > 2$]{\boldmath
Isomorphy Classes of $k$-involutions of $\So(n, k,\beta)$, $n > 2$}
\author{Robert W. Benim}
\address {Department of Mathematics and Computer Science\\
Pacific University\\
Forest Grove, OR, 97116}
\email{rbenim@gmail.com}
\author{Christopher E. Dometrius}
\address{Department of Mathematics\\
Lenoir-Rhyne University\\
Hickory, N. C., 28601} 
\email{chris.dometrius@lrc.edu}
\author{Aloysius G. Helminck}
\thanks{Third author is partially supported by N.S.F. Grant DMS-1063010 and NSA grant H98230-10-1-0252}
\address{Department of Mathematics\\
North Carolina State University\\
Raleigh, N. C., 27695} \email{loek@math.ncsu.edu}
\author{Ling Wu}
\address {Department of Mathematics\\
North Carolina State University\\
Raleigh, N. C., 27695}
\email{ling\_wu@hotmail.com}


\begin{document}
\maketitle

\begin{abstract}
A first characterization of the isomorphism classes of $k$-involutions for any reductive algebraic group defined over a perfect field was given in \cite{Helm2000} using $3$ invariants.  In \cite{HWD04,Helm-Wu2002} a full classification of all $k$-involutions on $\Sl(n,k)$ for $k$ algebraically closed, the real numbers, the $p$-adic numbers or a finite field was provided. In this paper, we find analogous results to develop a detailed characterization of the $k$-involutions of $\So(n,k,\beta)$, where $\beta$ is any non-degenerate symmetric bilinear form and  $k$ is any field not of characteristic $2$. We use these results to classify the isomorphy classes of $k$-involutions of $\So(n, k,\beta)$ for some bilinear forms and some fields $k$.

\end{abstract}

\section{Introduction}
Let $G$ be a connected reductive algebraic group defined over a
field $k$ of characteristic not $2$, $\theta$ a $k$-involution of $G$, i.e. a 
$k$-automorphism of $G$ of order two, 
$H$ a $k$-open subgroup of the fixed point group
of $\theta$ and $G_k$ (resp. $H_k$) the set of $k$-rational points
of $G$ (resp. $H$). The variety $G/H$ is called a symmetric variety and $G_k/H_k$ is called a symmetric
$k$-variety. 
For $k$ the real numbers or the $p$-adic numbers, the
symmetric $k$-varieties are also called reductive (real or $p$-adic) symmetric spaces or simply (real or $p$-adic)
symmetric spaces. 
Symmetric spaces play an important role in many areas of mathematics and physics, but probably best known are the representations associated with these symmetric spaces which have been studied  by many prominent mathematicians starting with a study of compact groups and their representations by Cartan \cite{Cartan29}, to a study of Riemannian symmetric spaces and real (and $p$-adic) groups by Harish Chandra \cite{Harish84} to a more recent study of the non Riemannian symmetric spaces (see for example \cite{Faraut79,Flensted-Jensen80,Oshima-Sekiguchi80,Ban-Schlichtkrull97,Ban-Schlichtkrull97b}) leading to a Plancherel formula in 1996 by Delorme \cite{Delorme96}. Once this Plancherel formula was obtained the attention shifted to $p$-adic symmetric spaces (see for example \cite{Delorme-Pascale14,Carmona-Delorme14,Delorme13,Helm-Helm02b,Helm-Helm05}).  
In the late 1980's generalizations of these  reductive symmetric spaces to other base fields started to play a role in other areas,  like in the study of arithmetic subgroups  (see \cite{Tong-Wang}),  the study of character sheaves (see for example 
\cite{Lusztig90a,Grojnowski92}), geometry (see
\cite{Procesi-Concini83,Procesi-Concini85} and
\cite{abeasis}), singularity theory (see \cite{Lusztig-Vogan83} and
\cite{Hirzebruch-Slodowy90}),   and the study of Harish Chandra modules (see
\cite{Beilinson-Bernstein81} and \cite{Vogan83,Vogan82}).  This prompted  Helminck and Wang to commence a study of rationality properties of  symmetric $k$-varieties over general base fields, see \cite{Helm-Wang93} for some first results.

To study these symmetric $k$-varieties for any field $k$, one needs  a classification. The isomorphy of these symmetric $k$-varieties can be reduced to the $\Inn (G,G_k)$-isomorphy  of the related  $k$-involutions (see \cite{Helm-Wang93}). Here $\Inn (G,G_k)$ is the set of inner automorphisms of $G$ that leave $G_k$ invariant. 
A characterization of these isomorphism
classes of the $k$-involutions was given in \cite{Helm2000}
essentially using the following 3 invariants:
\begin{enumerate}
\item classification of admissible $(\Gamma, \theta)$-indices.
\item classification of the $G_k$-isomorphism classes of
$k$-involutions of
 the $k$-an\-iso\-tro\-pic kernel of $G$.
\item classification of the $G_k$-isomorphism classes of $k$-inner
elements of $G$.
\end{enumerate}
For more details, see \cite{Helm2000}. The admissible $(\Gamma,
\theta)$-indices determine most of the fine structure of the
symmetric $k$-varieties and a classification of these was included
in \cite{Helm2000} as well. For $k$ algebraically closed or $k$ the
real numbers the full classification can be found in \cite{Helm88}.
For other fields a full classification of the remaining two invariants is
still lacking. In particular the case of symmetric $k$-varieties
over the $p$-adic numbers is of interest. We note that
the above characterization was only proven for $k$ a perfect field.

In \cite{HWD04,Helm-Wu2002} a full characterization of the $\Inn (G,G_k)$-isomorphism
classes of $k$-involutions was given in the case that $G=\Sl(n, k)$  which 
does not depend on any of the results in
\cite{Helm2000}.  It was also shown how one may construct an outer-$k$-involution from a given non-degenerate symmetric or skew-symmetric bilinear form $\beta$ of $k^n$.  Using this characterization the
possible isomorphism classes for $k$ algebraically
closed, the real numbers, the $p$-adic numbers and finite
fields were classified.

We note that much work has been done is characterizing automorphisms of $\oo(n,k,\beta)$, notably the works of Diedonne in \cite{Die51} and \cite{Die63}, Rickart in \cite{Ric50} and \cite{Ric51}, and Wonenburger in \cite{Won62}. These results consider automorphisms of $G_k$ under the action of the full automorphism group of $G_k$, so these results do not directly apply to the problem of isomorphy of $k$-involutions under the  action of the subgroup $\Inn (G,G_k)$ of $\aut (G_k)$ as considered in this paper. 

In this paper we study $k$-involutions of $\So(n, k, \beta)$, the special orthogonal group with respect to a symmetric bilinear form $\beta$ on $k^n$.  We give a characterization 
of the isomorphy classes of $k$-involutions of $\So(n, k, \beta)$, which come from inner automorphisms $\Inn_A$ of the general linear group $\Gl (n,\overline{k})$.  In section 2, we state many of the important definitions and initial observations. 
In section 3 we show that if a $k$-automorphism $\theta= \Inn_A$ where $A \in \Gl (n,\overline{k})$ leaves $\So(n,k,\beta)$ invariant, then 
we can assume $A$ is in $\So(n,k[\sqrt{\alpha}],\beta)$ when $n$ is odd, and  $A$ is in $\oo(n,k[\sqrt{\alpha}],\beta)$ when $n$ is even, where $k[\sqrt{\alpha}]$ is a quadratic extension of $k$. Further, we show that each entry of $A$ must be a $k$-multiple of $\sqrt{\alpha}$. To prove these results, we require that either $\chr(k) \ne 3$, or we require a restriction on the bilinear form $\beta$. In section 4, we have the main results of the paper. We determine which $A \in \So(n,k[\sqrt{\alpha}],\beta)$ if $n$ is odd and which $A \in \oo(n,k[\sqrt{\alpha}],\beta)$ if $n$ is even induce $k$-involutions of $\So(n,k,\beta)$ of the form $\Inn_A$. We will see that when $n$ is odd, there is only one type of these $k$-involutions, and if $n$ is even, there are four types of these $k$-involutions. For each type, we determine nice conditions that are equivalent to isomorphy of these $k$-involutions over $\So(n,k,\beta)$ if $n$ is odd and over $\oo(n,k,\beta)$ if $n$ is even. When $n \ne 3, 4, 6, 8$, the $k$-involutions of the form $\Inn_A$ will be all of the $k$-involutions of $\So(n,k,\beta)$. In section 5, we discuss the maximal number of possible isomorphy classes of $k$-involutions of the form $\Inn_A$. In section 6 we look at some explicit examples of orthogonal groups, most of which are the standard orthogonal group, when $k$ is algebraically closed, the real numbers, a finite field of order odd $q= p^m$ where $p>2$, or the $p$-adic numbers

\section{Preliminaries}
Our basic reference for
reductive groups will be the papers of Borel and Tits
\cite{Borel-Tits65}, \cite{Borel-Tits72} and also the books of
Borel \cite{Borel91},  Humphreys \cite{Humph75} and Springer
\cite{Spring81}. We shall follow their notations and terminology.
All algebraic groups and algebraic varieties are  taken over an
arbitrary field $k$ (of characteristic $\neq 2$) and all algebraic
groups considered are linear algebraic groups.

Our main reference for results regarding $k$-involutions of $\Sl(n,k)$ will be \cite{HWD04}.  Let
$k$ be a field of characteristic not $2$, $\bar k$ the algebraic closure of $k$,
$$\M(n,k)=\{ n\times n \text{-matrices with
entries in $k$} \}, $$
$$\Gl(n,k)= \{ A\in \M(n,k)\mid \det
(A)\neq 0\}$$  and
$$\Sl(n,k)= \{ A\in \M(n,k)\mid \det
(A)=1\}. $$  Let $k^*$ denote the multiplicative group of all the nonzero
elements, $(k^*)^2=\{a^2\mid a\in k^*\}$ and $I_n \in \M(n,k)$
denote the identity matrix. We will sometimes use $I$ instead of $I_n$ when the dimension of the identity matrix is clear.

We recall some important definitions and theorems from \cite{HWD04}.
\begin{definit}
\label{isoinv} 

Let $G$ be an algebraic groups defined over a field $k$. Let $G_k$ be the $k$-rational points of $G$. Let $\aut(G, G_k)$ denote the the set of $k$-automorphisms of $G_k$. That is, $\aut(G, G_k)$ is the set of automorphisms of $G$ which fix $G_k$. We say $\theta \in \aut(G,G_k)$ is a {\it $k$-involution} if $\theta^2 = \id$ but $\theta \ne \id$. That is, a $k$-involution is a $k$-automorphism of order 2.

Choose $A \in G_k$. Then the map $\Inn_A(X) = A^{-1}XA$ is called an {\it inner $k$-automorphism of $G_k$}. We denote the set of such $k$-automorphisms by $\Inn(G_k)$. If $\Inn_A \in \Inn(G_k)$ is a $k$-involution, then we say that $\Inn_A$ is an {\it inner $k$-involution of $G_k$}.

Assume $H$ is an algebraic group defined over $k$ which contains $G$. Let $H_k$ be the $k$-rational points of $H$. Choose $A \in H$. If the map $\Inn_A(X) = A^{-1}XA$ is such that $\Inn_A \in \aut(G,G_k)$, then we say that $\Inn_A$ is an {\it inner $k$-automorphism of $G_k$ over $H$}. We denote the set of such $k$-automorphisms by $\Inn(H,G_k)$. If $\Inn_A \in \Inn(H,G_k)$ is a $k$-involution, then we say that $\Inn_A$ is an {\it inner $k$-involution of $G_k$ over $H$}.

Suppose $\theta, \tau \in \aut((G,G_K)$.We say that $\theta$ is {\it isomorphic} to $\tau$ {\it over $H_k$} if there is $\phi$ in ${\Inn(H_k)}$ such that $\tau=\phi ^{-1}\theta\phi$. Equivalently,  we say that $\tau$ and $\theta$ are in the same \textit{isomorphy class over $H_k$}.

\end{definit}

In \cite{HWD04}, the isomorphy classes of the inner-$k$-involutions of $\Sl(n,k)$ were classified, and they are as follows:

\begin{theorem} \label{sltheorem1}
Suppose the $k$-involution $\theta \in \aut(\Sl(n,k))$ is of inner type. Then
up to isomorphism over $\Gl(n,k)$, $\theta$ is one of the following:
\begin{enumerate}
\item \label{sltheorem1.1} $\Inn_Y|_G$, where $Y = I_{n-i,i} \in
\Gl(n,k)$ where $i \in \left\{1, 2, \dots, \lfloor  \frac
{n}{2}\rfloor \right\}$ where $$I_{n-i,i} =  \left(\begin{array}{cc}I_{n-i} & 0 \\0 & -I_i\end{array}\right)$$. 

\item \label{sltheorem1.2} $\Inn_Y|_G$,
where
$Y =  L_{\frac n 2, x} \in \Gl(n,k)$ where $x$ is a fixed element of a coset of $k^*\slash
k^{*2}$, $x \not\equiv 1\mod k^{*2}$ and
$$L_{\frac{n}{2},x}=\begin{pmatrix} 0 & 1 & \hdots & 0 & 0 \\
x & 0 & \hdots & 0 & 0
\\ \vdots & \vdots & \ddots &
\vdots & \vdots \\ 0 & 0 & \hdots & 0 & 1\\ 0 & 0 & \hdots & x & 0
\end{pmatrix} .$$
 \end{enumerate}
Note that $(ii)$ can only occur when $n$ is even.
\end{theorem}

For the purposes of this paper, we will use matrices of the form 
$\left(\begin{smallmatrix}0 & I_{\frac{n}{2}} \\xI_{\frac{n}{2}} & 0\end{smallmatrix}\right)$ 
(and their multiples) rather than $L_{\frac n 2, x}$. Either of these serves as a member of the isomorphy class listed in the previous theorem. It will become apparent that the isomorphy classes of $k$-involutions $\So(n,k,\beta)$ over $\oo(n,k,\beta)$ are just isomorphy classes of $\Sl(n,k)$ over $\Gl(n,k)$ that have been divided into multiple isomorphy classes.

We now begin to define the notion of a special orthogonal group. Let $M$ be the matrix of a non-degenerate bilinear form $\beta$
over $k^n$ 
with respect to a basis $\{ e_1, \dots e_n \}$ of $k^n$.
We will say that $M$ is the matrix of $\beta$ if 
the basis $\{ e_1, \dots e_n \}$ 
is the standard basis of $k^n$. 

The typical notation for the orthogonal group is $\oo(n,k)$, which is the group $$\oo(n,k)= \{ A\in \M(n,k)\mid (Ax)^T(Ay) = x^Ty\}.$$ This group consists of the matrices which fix the standard dot product. This can be generalized to any non-degenerate symmetric bilinear $\beta$, which will yield the group $$\oo(n,k,\beta)= \{ A\in \M(n,k)\mid \beta(Ax,Ay) = \beta(x,y) \}.$$ If $M$ is the matrix of $\beta$ with respect to the standard basis, then we can equivalently say $$\oo(n,k,\beta)= \{ A\in \M(n,k)\mid A^TMA = M \}.$$ It is clear from this definition that all matrices in $\oo(n,k,\beta)$ have determinant 1 or -1. We define the {\it special orthogonal group of $\beta$} to be the group $$\So(n,k, \beta)= \oo(n,k, \beta) \cap \Sl(n,k),$$ and we define the {\it group of similitudes of $\beta$} to be the group $$\Go(n,k,\beta) = \{ A\in \M(n,k)\mid \beta(Ax,Ay) = \alpha \beta(x,y), \alpha \in k^* \}.$$ We note that $$\Go(n,k,\beta)= \{ A\in \M(n,k)\mid A^TMA = \alpha M, \alpha \in k^* \}.$$

We say two $n \times n$ matrices $A$ and $B$ are considered {\it congruent} over $k$ if there exists $Q \in \Gl(n,k)$ such that $Q^TAQ = B$. We also say that $A$ and $B$ are {\it congruent} via $Q$.

We note a couple of important facts, the first of which will be used repeatedly throughout this paper. 
\begin{enumerate}

\item Symmetric matrices are congruent to diagonal matrices when $\chr(k) \ne 2$, where the entries of the diagonal matrix are are representatives of the cosets of $k^*/(k^*)^2.$

\item If $\beta_1$ and $\beta_2$ correspond to $M_1$ and $M_2$, then $\So(n,k,\beta_1)$ and $\So(n,k,\beta_2)$ are isomorphic via 
$$\Phi:\So(n,k,\beta_1) \rightarrow \So(n,k,\beta_2): X \rightarrow Q^{-1}XQ$$ for some $Q \in \Gl(n,k)$ if and only if $Q^TM_1Q = M_2$ ($M_1$ and $M_2$ are congruent via $Q$).

\end{enumerate}

So, we will assume that $\beta$ is such that $M$ is diagonal. Then, to classify the $k$-involutions of an orthogonal group where $M$ is not diagonal, one can apply the characterization that will follow by simply using the isomorphism $\Phi$ given above.

We say two vectors $x,y\in k^n$ are said to be \textit{orthogonal} with respect to the bilinear form $\beta$ if $\beta(x,y)=0$. We will eventually see that orthogonal vectors play an important role in the structure of $k$-involutions of $\So(n,k,\beta)$. 

Lastly, we will always assume, whether stated or not, that $n > 2$ and that $\chr(k) \ne 2$.

\section{$k$-automorphisms of $\So(n,k,\beta)$}

In this paper, we consider the $k$-involutions that lie in the group $\Inn(\Gl(n,\overline{k}),\So(n,k,\beta))$. When $n$ is odd, this group turns out to be $\Inn(\So(n,\overline{k},\beta),\So(n,k,\beta))$, and when $n$ is even, this group turns out to be $\Inn(\oo(n,\overline{k},\beta),\So(n,k,\beta))$. We will see that when $n$ is odd and $n \ne 3$, that all of the $k$-involutions of $\So(n,k,\beta)$ will be inner $k$-involutions of $\So(n,k,\beta)$ over $\So(n, \overline{k}, \beta)$, and when $n$ is even and $n \ne 4, 6, 8,$, that all of the $k$-involutions of $\So(n,k,\beta)$ will be inner $k$-involutions of $\So(n,k,\beta)$ over $\oo(n, \overline{k}, \beta)$.  We will show isomorphy conditions of these $k$-involutions over $\So(n,k,\beta)$ when $n$ is odd, and over $\oo(n,k,\beta)$ when $n$ is even.

It follows from a proposition on page 191 of \cite{Borel91} that the outer $k$-automorphism group $$\out(\So(n,\overline{k},\beta)) = \aut(\So(n,\overline{k},\beta)) /\Inn(\So(n,\overline{k},\beta))$$ must be a subgroup of the diagram automorphisms of the associated Dynkin diagram. If $n = 2m+1$ and $m \ge 2$, then this Dynkin diagram is $B_m$ which has only the trivial diagram automorphism. Thus, there are no outer $k$-automorphisms of $\So(n,\overline{k},\beta)$ when $n$ is odd and $n \ge 5$. If $n = 2m$ and $m \ge 4$, then this Dynkin diagram is $D_m$. The group of automorphisms of this Dynkin diagram is $\mathbb{Z}_2$ when $m > 4$. So, when $n$ is even and $n \ge 10$, $\out(\So(n,\overline{k},\beta)) = \mathbb{Z}_2$. We will see that the outer $k$-automorphisms are of the form $\Inn_A$ where $A \in \oo(n,\overline{k},\beta)$ and $\det(A) = -1$. When $k$ is not algebraically closed, then all $k$-automorphisms of $\So(n,k,\beta)$ will still be of the form $\Inn_A$ for some $A \in \oo(n,\overline{k},\beta)$ since $\Inn_A$ must also be an $k$-automorphism of $\So(n,\overline{k},\beta).$ Thus, the characterizations that follow in this paper consider all $k$-automorphisms and $k$-involutions of $\So(n,k,\beta)$, assuming that $n \ne 3, 4, 6, 8$. In the cases where $n = 3, 4, 6, 8$, the results that follow only consider the $k$-automorphisms and $k$-involutions that can be written as $\Inn_A$ for some $A \in \Gl(n,\overline{k})$.

We now examine which $k$-automorphisms will act as the identity on $\So(n,k, \beta)$. This will prove to be useful when we classify matrix representatives for $k$-automorphisms. 

\begin{lem}
Assume $n > 2$. Let $A \in \Gl(n,\overline{k})$. If $\Inn_A$ is the identity on $\So(n,k, \beta)$, then $A $ is a diagonal matrix.
\end{lem}

\begin{proof}
Suppose $A$ is such that $\Inn_A$ is the identity on $\So(n,k, \beta)$. For $1 \le r < s \le n$, let $X_{rs}$ be the diagonal matrix with all 1's, except in the $r$th and $s$th diagonal entries, where instead there are -1's. This matrix always lies in $\So(n,k, \beta)$. So, we must have $A X_{rs} = X_{rs}A.$ On the left side, the matrix is the same as $A$, but with the $r$th and $s$th columns negated. On the right side, the matrix is the same as $A$, but with the $r$th and $s$th rows negated. So, all entries of $A$ on these rows and columns which aren't in the $(r,r)$, $(r,s)$, $(s,r)$ or $(s,s)$ components must be equal to 0, since this is the only number which equals its negative. To see that the $(r,s)$ and $(s,r)$ components of $A$ must also equal 0, we can repeat this process for $X_{rt}$, where $t$ is distinct from both $r$ and $s$. (Note that this is where we use the fact that $n > 2$.) Thus, all off-diagonal elements of $A$ are 0, which means $A$ is diagonal.
\end{proof}


We want to be able to say more about the matrix $A$ when $\Inn_A$ acts as the identity. It turns out that if we make the following assumption on the orthogonal group $\So(n,k,\beta)$, then we can show that $A$ is a multiple of the identity. 

\begin{definit}
Let $k$ be a field and suppose $\beta$ is a bilinear form on $k^n$ such that it has matrix representation $M$, where $M$ is diagonal with diagonal entries $m_1,...,m_n$, which are representatives in $k^*$ of cosets of $k^*/(k^*)^2$. If for each pair $m_s$ and $m_t$, $x^2+\frac{m_s}{m_t}y^2 = 1$ has a solution $(x,y)$ such that $y \ne 0$, then we call $\So(n,k,\beta)$ a {\it friendly orthogonal group}. 
\end{definit}

With this new terminology in mind, we get the following result.

\begin{lem}
\label{IdentityLem}
Assume $n > 2$. Suppose $\So(n,k, \beta)$ is a friendly orthogonal group. Let $A \in \Gl(n,\overline{k})$. Then, $\Inn_A$ is the identity on $\So(n,k, \beta)$ if and only if $A = \alpha I$ for some $\alpha \in \overline{k}^*$.
\end{lem}

\begin{proof}

We know from the previous lemma that $A$ is diagonal. Let $a_i$ represent the $i$th diagonal entry of $A$. Recall that we are assuming that $M$ is diagonal. Let $m_i$ represent the $i$th diagonal entry of $M$. Then, there exists $a, b \in k$ where $b \ne 0$ such that $a^2+\frac{m_s}{m_t}b^2 = 1.$ For $1 \le i < j \le n$, let $$Y_{st} = \left(\begin{array}{cccccccccccc}1 & 0 & \cdots &   &   &   &   &   &   &   & \cdots & 0 \\0 & 1 &   &   &   &   &   &   &   &   &   & \vdots \\\vdots &   & \ddots &   &   &   &   &   &   &   &   &   \\  &   &   & 1 &   &   &   &   &   &   &   &   \\  &   &   &   & a & 0 & \cdots & 0 & b &   &   &   \\  &   &   &   & 0 & 1 &   &   & 0 &   &   &   \\  &   &   &   & \vdots &   & \ddots &   & \vdots &   &   &   \\  &   &   &   & 0 &   &   & 1 & 0 &   &   &   \\  &   &   &   & -\frac{m_s}{m_t}b & 0 & \cdots & 0 & a &   &   &   \\  &   &   &   &   &   &   &   &   & 1 &   & \vdots \\\vdots &   &   &   &   &   &   &   &   &   & \ddots & 0 \\0 & \cdots &   &   &   &   &   &   &   & \cdots & 0 & 1\end{array}\right),$$ where the noteworthy entries occur in the $s$th and $t$th rows and columns. It is a simple calculation to show that $Y_{st}^TMY_{st} = M$, and that $\det(Y_{st}) =1$. So, $Y_{st} \in \So(n,k, \beta)$. Then, we know that $AY_{st} = Y_{st}A.$ By comparing both sides of this equality and inspecting the $(s,t)$ entry, we see that $ba_t = ba_s$. Since we are assuming that $b \ne 0$, then it follows that $a_t = a_s$. Since we can repeat this for all $s$ and $t$, then it is clear that $A$ is a multiple of the identity.

\end{proof}

This result is only useful if we can show that $\So(n,k,\beta)$ is commonly a friendly orthogonal group. In the following theorem, we see that most $\So(n,k,\beta)$ are friendly.

\begin{lem}
\begin{enumerate}

\item If $\chr(k) \ne 2, 3$, then $\So(n,k,\beta)$ is a friendly orthogonal group. 

\item If $M = \left(\begin{array}{cccc}m_1 &   &   &   \\  & m_2 &   &   \\  &   & \ddots &   \\  &   &   & m_n\end{array}\right)$ is such that $m_s \ne -m_t$ whenever $s \ne t$ and $\chr(k) \ne 2$, then $\So(n,k,\beta)$ is a friendly orthogonal group. 

\end{enumerate}
\end{lem}

\begin{proof}

When $\chr(k) \ne 2$, then we see that $1 = x^2+\alpha y^2$ has solution $(x,y) = \left(\frac{\alpha-1}{\alpha+1}, \frac{2}{\alpha+1}\right)$ when $\alpha \ne -1$. When $\chr(k) \ne 3$, then we see that $1=x^2-y^2$ has solution $(x,y) = \left( \frac{5}{3}, \frac{4}{3} \right)$. Based on these two solutions, it is clear that $x^2+\frac{m_s}{m_t}y^2 = 1$ will always have a solution in $k$ if 
$\chr(k) \ne 2, 3$, and also when $\frac{m_s}{m_t} \ne -1$ and $\chr(k) \ne 2$.
\end{proof}

To show that this condition on orthogonal groups is not trivial, we note a case where $\So(n,k, \beta)$ is not a friendly orthogonal group.

\begin{beisp}
Suppose $k = \mathbb{F}_3$ and that $\beta$ is such that $M = \left(\begin{array}{cc}I & 0 \\0 & -1\end{array}\right).$ Then, $\So(n,\mathbb{F}_3, \beta)$ is not a friendly orthogonal group, because there is no solution to $x^2-y^2=1$ where $y \ne 0.$
\end{beisp}

For the remainder of the paper, we will assume that all orthogonal groups are friendly.

Fix a bilinear form $\beta$ with matrix $M$. If $A \in \Gl(n,k)$ is a matrix such $M^{-1}A^TMA = \alpha  I_{n}$, then $A \in \Go(n,k,\beta) = \{ A\in \M(n,k)\mid \beta(Ax,Ay) = \alpha \beta(x,y), \alpha \in k \}$, the group of similitudes of $\beta$.

We now have the following preliminary result that characterizes $k$-automorphisms of $\So(n,k, \beta)$.

\begin{lem}

Assume $n >2$. If $A \in \Gl(n,\overline{k})$, then $\Inn_A(\So(n,k,\beta)) \subseteq\So(n,k, \beta)$ if and only if $A \in \Go(n,\overline{k},\beta)$, $A = p \widetilde{A}$ where $p \in \overline{k}$, and 
\begin{enumerate}
\item $\widetilde{A} \in \So(n, \overline{k}, \beta)$ if $n$ odd, or
\item $\widetilde{A} \in \oo(n, \overline{k}, \beta)$ if $n$ is even.
\end{enumerate}

\end{lem}

\begin{proof}

Suppose $A \in \Gl(n,\overline{k})$ and $\Inn_A(\So(n,k, \beta)) \subseteq \So(n,k, \beta)$. Choose $X \in \So(n,k, \beta).$ Then, $A^{-1}XA \in \So(n,k, \beta).$ So, 
\begin{align*}
I &= (A^{-1}XA )^{-1}A^{-1}XA \\ & = M^{-1}(A^{-1}XA )^TMA^{-1}XA \\ & =  M^{-1}A^{T}X^T(A^{-1})^T MA^{-1}XA. 
\end{align*}
This implies that $$A^{-1}X^{-1} = M^{-1}A^{T}X^T(A^{-1})^T MA^{-1},$$ which means $$X^{-1} = AM^{-1}A^{T}X^T(A^{-1})^T MA^{-1}.$$ We can rewrite this as $$M^{-1}X^TM = AM^{-1}A^{T}X^T(A^{-1})^T MA^{-1}.$$ If we transpose both sides, then we see that $$MXM^{-1} = (A^{-1})^TMA^{-1}XAM^{-1}A^T.$$ Solving for the $X$ term on the left, we get that 
\begin{align*}
X &= M^{-1}(A^{-1})^TMA^{-1}XAM^{-1}A^TM \\ &= (AM^{-1}A^TM)^{-1}X(AM^{-1}A^TM).\end{align*}
By Lemma \ref{IdentityLem}, it follows that $(AM^{-1}A^TM) = \alpha I$ for some $\alpha \in \overline{k}^*.$ Thus, $A \in \Go(n,\overline{k},\beta)$. Choose $p \in \overline{k}$ such that $p^2 = \alpha$. Then let $ \widetilde{A} = \frac{1}{p}A.$  It follows that $M^{-1}\widetilde{A}^TM \widetilde{A} = \frac{1}{p^2}M^{-1}A^TMA = \frac{1}{\alpha}(\alpha  I) =  I, $ which shows that $\widetilde{A}$ is orthogonal. That is, $\widetilde{A} \in \oo(n, \overline{k}, \beta)$. If $n$ is odd and $\det(\widetilde{A}) = -1$, then we can replace $\widetilde{A}$ with $-\widetilde{A}$, and instead have a matrix inside $\So(n, \overline{k}, \beta)$.

Since the converse  is clear, we have proven the statement.
\end{proof}

In the following theorem which completes the characterization of $k$-automorphisms on $\So(n,k, \beta)$, we see that we can choose the entries of $A$ to be in $k$ or a quadratic extension of $k$. For the remainder of this paper, we will use $\sqrt{\alpha}$ to denote a fixed square root of $\alpha \in k^*$.

\begin{theorem}
\label{CharThm2So}
Assume $n >2$.
\begin{enumerate}
\item If $n$ is odd and $A$ is in $\So(n,\overline{k}, \beta)$, then $\Inn_A$  keeps $\So(n,k, \beta)$ invariant if and only if  we can choose $\tilde{A} \in \So(n,k,\beta)$ such that $\Inn_{\tilde{A}} = \Inn_A$.
\item If $n$ is even and $A$ is in $\oo(n,\overline{k}, \beta)$, then $\Inn_A$  keeps $\So(n,k, \beta)$ invariant if and only if there exists $p \in \overline{k}$ and $B \in \Go(n,k,\beta)$ such that $B = pA$. Further, we can show $A \in \oo(n, k[\sqrt{\alpha}], \beta)$ where each entry of $A$ is a $k$-multiple of $\sqrt{\alpha}.$ 
\end{enumerate}
\end{theorem}

\begin{proof}

Let $n >2$ be arbitrary, and suppose $A$ is in $\oo(n,\overline{k},\beta)$ such that $\Inn_A$  keeps $\So(n,k, \beta)$ invariant. Let $a_{ij}$ be the $(i,j)$ entry of $A$. We break the proof into steps. At the beginning of each step, we state what we shall prove in that step.

{\bf Step One:} $a_{ri}a_{rj}+a_{si}a_{sj} \in k$ for all $i,j, r,s$.

Let $X_{rs}$ be the diagonal matrix with all entries -1, except for the $(r,r)$ and $(s,s)$ entries, which are 1. Since $M$ is diagonal, it is clear that $X_{rs} \in \oo(n,k, \beta)$. If $n$ is even, then $X_{rs} \in \So(n,k,\beta)$. If $n$ is odd, then $-X_{rs} \in \So(n,k,\beta)$. So, we know that $\Inn_A(X_{rs})$ or $-\Inn_A(X_{rs})$ must lie in $\So(n,k,\beta)$. It is also clear that $\Inn_A(I) \in \So(n,k, \beta)$. So, both $\Inn_A(X_{rs})$ and $\Inn_A(I)$ have entries in $k$. Let us examine the entries of $\Inn_A(X_{rs})$:
\begin{align*}
\Inn_A(X_{rs}) &= A^{-1}X_{rs}A = (M^{-1}A^T)(MX_{rs}A)\\
&= \left(\begin{array}{ccc}\frac{a_{11}}{m_1} & \cdots & \frac{a_{n1}}{m_1} \\\vdots & \ddots & \vdots \\\frac{a_{1n}}{m_n} & \cdots & \frac{a_{nn}}{m_n}\end{array}\right)  \left(\begin{array}{ccc}-m_1a_{11} & \cdots & -m_1a_{1n} \\\vdots &   & \vdots \\-m_{r-1}a_{r-1,1} & \cdots & -m_{r-1}a_{r-1,n} \\m_ra_{r1} & \cdots & m_ra_{rn} \\-m_{r+1}a_{r+1,1} & \cdots & -m_{r+1}a_{r+1,n} \\\vdots &   & \vdots \\-m_{s-1}a_{s-1,1} & \cdots & -m_{s-1}a_{s-1,n} \\m_sa_{s1} & \cdots & m_sa_{sn} \\-m_{s+1}a_{s+1,1} & \cdots & -m_{s+1}a_{s+1,n} \\\vdots &   & \vdots \\-m_{n}a_{n1} & \cdots & -m_{n}a_{nn}\end{array}\right)\\
&= \left( a_{ri}a_{rj}+a_{si}a_{sj}-\sum_{l \ne r,s}a_{li}a_{lj} \right)_{(i,j)}.
\end{align*}

Since $\Inn_A(X_{rs})$ and $I= \Inn_A(I)$ have entries in $k$, then so does the matrix $\Inn_A(X_{rs}+I) = I+\Inn_A(X_{rs}).$ Using a similar calculation to the above, we can see that $\Inn_A(I)+\Inn_A(X_{rs})$ has entries of the form $2a_{ri}a_{rj}+2a_{si}a_{sj}.$ It follows that $a_{ri}a_{rj}+a_{si}a_{sj} \in k$ for all $i,j, r,s$.

{\bf Step Two:} $a_{ri}a_{rj} \in k$ for all $i,j,r$.

Choose integer $t$ distinct from $r$ and $s$ such that $1 \le t \le n$. We have that $$a_{ri}a_{rj}-a_{ti}a_{tj}  =(a_{ri}a_{rj}+a_{si}a_{sj})-(a_{si}a_{sj}+a_{ti}a_{tj}) \in k,$$ 
which means that $$a_{ri}a_{rj} = \frac{1}{2}(a_{ri}a_{rj}-a_{ti}a_{tj} )+\frac{1}{2}(a_{ri}a_{rj}+a_{ti}a_{tj} ) \in k,$$ for all $i,j,r$.

{\bf Step Three:} $a_{ir}a_{jr} \in k$ for all $i,j, r$. 

Now, we consider the bilinear form $\beta_1$ which has matrix $M^{-1}$ (using the standard basis for $k^n$). We know that $X \in \So(n,\overline{k},\beta)$ if and only $X^TMX = M$. But, if that is the case for a given $X$, then it follows that $$X^{-1}M^{-1}(X^{-1})^T= M^{-1}.$$ Thus, $(X^{-1})^T \in \So(n,\overline{k},\beta_1).$ Since this is a group, then we also know that $X^T \in \So(n,\overline{k},\beta_1).$ It is then easy to see that $X \in \So(n,k,\beta)$ if and only if $X^T \in \So(n,k,\beta_1)$.

We further claim that $\Inn_{A^T}$ is a $k$-automorphism of $ \So(n,k,\beta_1)$. Suppose $Y \in \So(n,k,\beta_1)$ and consider $\Inn_{A^T}(Y) = (A^T)^{-1}YA^T.$ This matrix lies in $ \So(n,k,\beta_1)$ if and only if its inverse-transpose lies in $ \So(n,k,\beta)$. The inverse-transpose of $\Inn_{A^T}(Y)$ is $A(Y^{-1})^TA^{-1} = \Inn_{A^{-1}}((Y^{-1})^T) = (\Inn_A)^{-1}((Y^{-1})^T).$ This proves our claim.

Since $\Inn_{A^T}$ is a $k$-automorphism of $ \So(n,k,\beta_1)$, then it follows Step Two that $a_{ir}a_{jr} \in k$ for all $i,j, r$. 

{\bf Step Four:} $a_{si}a_{tj}-a_{ti}a_{sj} \in k,$ for all $i,j,s,t$.

We now recall the matrices $Y_{st} \in \So(n,k,\beta)$ from the proof of Lemma \ref{IdentityLem}. So, it must be the case that $\Inn_A(Y_{st}) \in \So(n,k,\beta)$. Let us examine the entries of $\Inn_A(Y_{st})$:
\begin{align*}
\Inn_A(Y_{st}) &= A^{-1}Y_{st}A\\ &= (M^{-1}A^T)(MY_{st}A)\\ 
&= \left(\begin{array}{ccc}\frac{a_{11}}{m_1} & \cdots & \frac{a_{n1}}{m_1} \\\vdots & \ddots & \vdots \\\frac{a_{1n}}{m_n} & \cdots & \frac{a_{nn}}{m_n}\end{array}\right)  \left(\begin{array}{ccc}m_1a_{11} & \cdots & m_1a_{1n} \\\vdots &   & \vdots \\m_{s-1}a_{s-1,1} & \cdots & m_{s-1}a_{s-1,n} \\am_sa_{s1}+bm_sa_{t1} & \cdots & am_sa_{sn}+bm_sa_{tn} \\m_{s+1}a_{s+1,1} & \cdots & m_{s+1}a_{s+1,n} \\\vdots &   & \vdots \\m_{t-1}a_{t-1,1} & \cdots & m_{t-1}a_{t-1,n} \\-bm_ta_{s1}+am_ta_{t1} & \cdots & -bm_ta_{sn}+am_ta_{tn} \\m_{t+1}a_{t+1,1} & \cdots & m_{t+1}a_{t+1,n} \\\vdots &   & \vdots \\m_{n}a_{n1} & \cdots & m_{n}a_{nn}\end{array}\right)\\
&= \left( a_{si}(aa_{sj}+ba_{tj})+a_{ti}(-ba_{sj}+aa_{tj})+\sum_{l \ne s,t} a_{li}a_{lj} \right)_{(i,j)}.
\end{align*}

Since each of the matrix entries of $\Inn_A(Y_{ij})$ must lie in $k$ and we also know that $a_{li}a_{lj} \in k$ for $1 \le l \le n$ from Step Two, then it follows that 

$$a_{si}(aa_{sj}+ba_{tj})+a_{ti}(-ba_{sj}+aa_{tj})$$
$$=a(a_{si}a_{sj})+b(a_{si}a_{tj})-b(a_{ti}a_{sj})+a(a_{ti}a_{tj}) \in k.$$

We know $b \in k$, and that $a_{si}a_{sj}$ and $a_{ti}a_{tj} \in k$ by Step Two. So, we see that $a_{si}a_{tj}-a_{ti}a_{sj} \in k,$ for all $i,j,s,t$.
  
 {\bf Step Five:} $a_{si}a_{tj} \in k$ for all $i,j,s,t$. 
 
 {\bf Substep Five A} : $a_{si}(a_{si}a_{tj}-a_{ti}a_{sj}) = ca_{si}$ for $c \in k$.

If both $a_{si}$ and $a_{tj} \in k$, then $a_{is}a_{jt} \in k$ is clear. So, we will assume $a_{tj} \not \in k$. We will also assume that $a_{si} \ne 0$, since if $a_{si} = 0$, then it is again clear that $a_{si}a_{tj} \in k$.

From Step Four, we know that $$a_{si}a_{tj}-a_{ti}a_{sj} \in k.$$ It follows that $$a_{si}(a_{si}a_{tj}-a_{ti}a_{sj}) \in k[a_{si}].$$ Let $c = a_{si}a_{tj}-a_{ti}a_{sj}$. Then, we have $a_{si}(a_{si}a_{tj}-a_{ti}a_{sj}) = ca_{si}$ for $c \in k$.

 {\bf Substep Five B} : $a_{si}(a_{si}a_{tj}-a_{ti}a_{sj}) = da_{tj}$ for $d \in k$.

We now consider $a_{si}(a_{si}a_{tj}-a_{ti}a_{sj}) = a_{si}^2a_{tj}-a_{si}a_{sj}a_{ti}$ in a different fashion. From Step Two, we know that $a_{ti}a_{tj} \in k$. Also recall that we are assuming $a_{tj} \not \in k$. Since Step Two also tells us that $a_{tj}^2 \in k$, then it follows that $a_{ti}$ and $a_{tj}$ must lie in the same coset of $k^*/ (k^*)^2$. So, there exists $p \in k$ such that $a_{ti} = pa_{tj}$. From these observations, we see that $$a_{si}(a_{si}a_{tj}-a_{ti}a_{sj}) = (a_{si}^2-a_{si}a_{sj}p)a_{tj}.$$ From Step Two, we know that $a^2_{si}, a_{si}a_{sj} \in k$. Let $d = a_{si}a_{tj}-a_{ti}a_{sj}$, and note that $d \in k$. We have just shown that $a_{si}(a_{si}a_{tj}-a_{ti}a_{sj}) = da_{tj}$ for $d \in k$.

 {\bf Substep Five C:} $a_{si}a_{tj} \in k$ for all $i,j,s,t$. 
 
 Combining Substeps Five A and B, we see that $ca_{si} = da_{tj}$ for some $c, d \in k$. We prove this step by considering the two cases where $c \ne 0$ and $c = 0$.
 
 If $c \ne 0$, then $d \ne 0$. Further, we have that $a_{si} = \frac{d}{c}a_{tj}.$ From this, we see that $a_{si}a_{tj} = \frac{d}{c}a_{tj}^2.$ Since $\frac{d}{c} \in k$ by assumption and $a_{tj}^2$ by Step Two, then we have that $a_{si}a_{tj} \in k$.
 
 Alternatively, if $c = 0$, then since $c = a_{si}a_{tj}-a_{ti}a_{s}$, we know that $a_{si}a_{tj}=a_{ti}a_{sj}$. Recall that in Substep Five B that $a_{ti} = pa_{tj}$ for $p \in k$. Combining these facts, we have that 
 
 \begin{align*}
 a_{si}a_{tj} &= a_{ti}a_{sj}\\
 &= pa_{tj}a_{sj}\\
 \end{align*}
 
 By Step Three, we know that $a_{tj}a_{sj} \in k$. So, it follows that $a_{si}a_{tj} \in k$.

{\bf Step Six:} $A \in \oo(n,k[\sqrt{\alpha}],\beta)$ for some $\alpha \in k$.

From Step Five, it is clear that $k[a_{si}] = k[a_{tj}]$ for all $i,j,s,t$ (assuming that $a_{is}$ and $a_{jt}$ are both nonzero). Fix a nonzero entry $a_{si}$ of $A$. Let $\alpha = a_{is}^2$ and denote $a_{si}$ as $\sqrt{\alpha}$. We have shown that all the entries of $A$ are in $k[\sqrt{\alpha}].$ This means that $A \in \oo(n,k[\sqrt{\alpha}],\beta)$, and all of the entries of $A$ are $k$-multiples of $\sqrt{\alpha}$, as desired. 

{\bf Step Seven:} If $n$ is odd, then $A \in \So(n,k,\beta)$.

If $n$ is odd, then we can replace $A$ with $-A$ to get a matrix in $\So(n,k[\sqrt{\alpha}],\beta).$ So, assume that $A \in \So(n,k[\sqrt{\alpha}],\beta).$ We now show that we do not need a quadratic extension of $k$ when $n$ is odd. Proceed by contradiction and assume  $A \in \So(n,k[\sqrt{\alpha}],\beta)$ where $\sqrt{\alpha} \not \in k$. From Step Six we know that $\sqrt{\alpha}A \in \Gl(n,k)$. Then, 
$$\det(\sqrt{\alpha}A ) = [\sqrt{\alpha}]^n\det(A) = \alpha^{\frac{n}{2}} \not \in k,$$
which is a contradiction. So, if $n$ is odd, then $A \in \So(n,k,\beta)$.
\end{proof}

\section{$k$-involutions of $\So(n,k, \beta)$}

We now begin to focus on $k$-involutions and their characterization, as eel as the characterization of their isomorphy classes. We will distinguish different types of $k$-involutions. First, we note that for some $k$-involutions, $\phi$, there exists $A \in \oo(n,k,\beta)$ such that $\phi = \Inn_A$, but not in all cases. Sometimes we must settle for $A \in \oo(n,k[\sqrt{\alpha}], \beta) \setminus \oo(n,k,\beta)$ where each entry of $A$ is a $k$-multiple of $\sqrt{\alpha}$.

This is not the only way in which we can distinguish between different types of $k$-involutions. If $\Inn_A$ is a $k$-involution, then $\Inn_{A^2} = (\Inn_A)^2$ is the identity map. We know from earlier that this means that $A^2 = \gamma I$ for some $\gamma \in \overline{k}.$ But, we know for certain that $A$ is orthogonal. So, $A^2$ is also orthogonal. That means that $(A^2)^TM(A^2) = M$, which implies $(\gamma I)^TM (\gamma I) = M$, which means $\gamma^2 = 1$. So, $\gamma = \pm 1.$ Thus, we can also distinguish between different types of $k$-involutions by seeing if $A^2 = I$ or $A^2 = -I$. This gives the four types of $k$-involutions, which are outlined in Table \ref{InvDefSo}.
\begin{table}[h] 
\centering
\caption {The various possible types of $k$-involutions of $\So(n,k,\beta)$}  \label{InvDefSo}
\begin{tabular}[t]{|c||c|c|}
\hline  & $ A \in \oo(n,k,\beta)$ & $A \in \oo(n,k[\sqrt{\alpha}], \beta) \setminus \oo(n,k,\beta)$ \\
\hline  \hline $A^2 = I$ & Type 1 & Type 2 \\ 
\hline  $A^2 = -I$ & Type 3 & Type 4 \\
\hline
\end{tabular}
\end{table}
It follows from our characterization of $k$-automorphisms that when $n$ is odd, that Type 2 and Type 4 $k$-involutions do not occur. But, we also see that if $n$ is odd and $A$ is orthogonal, then $A^2$ must have determinant 1. So, we see in addition that Type 3 $k$-involutions can also only occur when $n$ is even.

In the following theorem, we show that an isomorphy class of $k$-involutions will lie neatly into exactly one of these types of $k$-involutions. First we need a lemma with a condition equivalent to isomorphism over $\oo(n,k,\beta)$.

\begin{lem}
\label{TidyLem}
Assume $\Inn_A$ and $\Inn_B$ are $k$-involutions of $\So(n,k,\beta)$ where $A$ lies in $\oo(n,k,\beta)$ or in $\oo(n,k[\sqrt{\alpha}],\beta)$ and each entry of $A$ is a $k$-multiple of $\sqrt{\alpha}$, and likewise for $B$. Then, $\Inn_A$ and $\Inn_B$ are isomorphic over $\oo(n,k,\beta)$ (or $\So(n,k,\beta)$) if and only if there exists $Q \in \oo(n,k,\beta)$ (or $\So(n,k,\beta)$) such that $Q^{-1}AQ = B$ or $-B$.
\end{lem}

\begin{proof}
 First assume there exists $Q \in \oo(n,k,\beta)$ (or $\So(n,k,\beta)$) such that $Q^{-1}AQ = B$ or $-B$. Then, for all $U \in \So(n,k,\beta)$, we have 
\begin{align*}
\Inn_Q\Inn_A \Inn_{Q^{-1}} (U) &= Q^{-1}A^{-1}QUQ^{-1}AQ\\ 
&= (Q^{-1}AQ)^{-1}U(Q^{-1}AQ)\\
&= (\pm B)^{-1}U(\pm B)\\ 
&= B^{-1}UB\\ 
&= \Inn_B(U).
\end{align*}
So, $\Inn_Q\Inn_A \Inn_{Q^{-1}} =\Inn_B$. That is, $\Inn_A$ and $\Inn_B$ are isomorphic over $\oo(n,k,\beta)$ (or $\So(n,k,\beta)$).

To prove the converse, we now assume that $\Inn_A$ and $\Inn_B$ are isomorphic over $\oo(n,k,\beta)$ (or $\So(n,k,\beta)$). So, there exists $Q \in \oo(n,k,\beta)$ (or $\So(n,k,\beta)$) such that $\Inn_Q\Inn_A \Inn_{Q^{-1}} =\Inn_B$. Thus, for all $U \in \So(n,k,\beta)$, we have 

$$Q^{-1}A^{-1}QUQ^{-1}AQ = B^{-1}UB,$$

which implies 

$$BQ^{-1}A^{-1}QUQ^{-1}AQB^{-1} = U.$$

So, $Q^{-1}AQB^{-1} $ commutes with all elements of $\So(n,k,\beta)$. The center of $\So(n,k,\beta)$ is $\{I \}$ if $n$ is odd and  $\{ I, -I \}$ if n is even. So, $Q^{-1}AQ = B$ or $-B$. 

\end{proof}

\begin{theorem}
\label{TypesAreTidy}
Assume $\Inn_A$ and $\Inn_B$ are $k$-involutions of $\So(n,k,\beta)$ where $A$ lies in $\oo(n,k,\beta)$ or in $\oo(n,k[\sqrt{\alpha}],\beta)$ and each entry of $A$ is a $k$-multiple of $\sqrt{\alpha}$, and likewise for $B$. If $\Inn_A$ and $\Inn_B$ are isomorphic over $\oo(n,k,\beta)$, then they must be $k$-involutions of the same type.
\end{theorem}

\begin{proof}

Assume $\Inn_A$ and $\Inn_B$ are isomorphic over $\oo(n,k,\beta)$. By Lemma \ref{TidyLem} we can assume there exists $Q \in \oo(n,k,\beta)$ such that $Q^{-1}AQ = B$ or $-B$. This implies that 
\begin{enumerate}
\item $A \in \oo(n,k,\beta)$ if and only if $B \in \oo(n,k,\beta)$; and

\item $A^2 = I$ if and only if $B^2 = I$.
\end{enumerate} 

Thus, $\Inn_A$ and $\Inn_B$ must be of the same type.
\end{proof}

We introduce the notation $E(A, \lambda)$ to refer to the eigenspace of a matrix $A$ corresponding to eigenvalue $\lambda$, where we assume the vectors lie in $\overline{k}^n$. In practice, we will be concerned with basis vectors that lie either in $k$, or a particular quadratic extension of $k$.

\subsection{Type 1 $k$-involutions}

We now find a structured form for the matrices of all types of $k$-involutions. We begin with Type 1 $k$-involutions. When $n$ is odd, these are the only $k$-involutions. 

\begin{lem}
\label{Type1ClassSo}
Suppose $\theta$ is a Type 1 $k$-involution of $\So(n,k,\beta)$. Then, there exists $A \in \oo(n,k,\beta)$ such that $A = X \left(\begin{smallmatrix}-I_s & 0 \\0 & I_t\end{smallmatrix}\right) X^{-1}$ where $s+t =n$, $s \le t$, and $$X  = \left(\begin{matrix}x_1 & x_2 & \cdots & x_n \end{matrix}\right)\in \Gl(n,k),$$ where the $x_i$ are orthogonal eigenvectors of $A$, meaning $X^TMX$ is diagonal.
\end{lem}
\begin{proof}
Since $A^2 = I$, then all eigenvalues of $A$ are $\pm 1$. Since there are no repeated roots in the minimal polynomial of $A$, then we see that $A$ is diagonalizable. We want to construct bases for $E(A,1)$ and $E(A,-1)$ such that all the basis vectors lie in $k^n$. Let $s = \dim(E(A,-1))$ and $t = \dim(E(A,1))$, and observe that $s+t = n$ since $A$ is diagonalizable. If $s >t$, then replace $A$ with $-A$, and use this matrix instead. (It will induce the same $k$-involution.) Let $\{z_1,...,z_n\}$ be a basis for $k^n$. For each $i$, let $u_i = (A-I)z_i.$ Note that $$Au_i = A(A-I)z_i = -(A-I)z_i = -u_i.$$ So, $\{u_1,...,u_n\}$ must span $E(A,-1)$. Thus, we can appropriately choose $s$ of these vectors and form a basis for $E(A,-1)$. Label these basis vectors as $y_1,...,y_s$. We can similarly form a basis for $E(A,1)$. We shall call these vectors $y_{s+1},...,y_n$. Let $Y$ be the matrix with the vectors $y_1,...,y_n$ as its columns. Then, by construction, $$Y^{-1}AY =  \left(\begin{array}{cc}-I_s & 0 \\0 & I_t\end{array}\right).$$ We can rearrange to get $$A = Y \left(\begin{array}{cc}-I_s & 0 \\0 & I_t\end{array}\right) Y^{-1}.$$
Recall that $A^T = MAM^{-1}$, since $A \in \oo(n,k,\beta)$. So, 
$$\left( Y \left(\begin{array}{cc}-I_s & 0 \\0 & I_t\end{array}\right) Y^{-1} \right)^T = M \left( Y \left(\begin{array}{cc}-I_s & 0 \\0 & I_t\end{array}\right) Y^{-1} \right) M^{-1}.$$
This implies 
$$(Y^{-1})^T \left(\begin{array}{cc}-I_s & 0 \\0 & I_t\end{array}\right) Y^T = MY \left(\begin{array}{cc}-I_s & 0 \\0 & I_t\end{array}\right) (MY)^{-1}$$
which means
$$ \left(\begin{array}{cc}-I_s & 0 \\0 & I_t\end{array}\right) Y^TMY = Y^TMY \left(\begin{array}{cc}-I_s & 0 \\0 & I_t\end{array}\right) .$$
So, $Y^TMY = \left(\begin{smallmatrix}Y_1 & 0 \\0 & Y_2\end{smallmatrix}\right)$, where $Y_1$ is $s \times s$, $Y_2$ is $t \times t$, and both are symmetric. It follows that there exists $N = \left(\begin{smallmatrix}N_1 & 0 \\0 & N_2\end{smallmatrix}\right) \in \Gl(n,k)$ such that $N^TY^TMYN$ is diagonal. Let $X = YN$. Then, 
\begin{align*}
 X \left(\begin{array}{cc}-I_s & 0 \\0 & I_t\end{array}\right) X^{-1} &= YN\left(\begin{array}{cc}-I_s & 0 \\0 & I_t\end{array}\right) (YN)^{-1}\\
&= Y \left(\begin{array}{cc}N_1 & 0 \\0 & N_2\end{array}\right) \left(\begin{array}{cc}-I_s & 0 \\0 & I_t\end{array}\right) \left(\begin{array}{cc}N_1^{-1} & 0 \\0 & N_2^{-1}\end{array}\right)Y^{-1}\\
&= Y  \left(\begin{array}{cc}-I_s & 0 \\0 & I_t\end{array}\right) Y^{-1}\\ 
&= A,
\end{align*}
 where $X^TMX$ is diagonal.  It follows from this last observation that the column vectors of $X$ must be orthogonal with respect to $\beta$.
\end{proof}

Now we show conditions equivalent to isomorphy. Note that we say two matrices are $A$ and $B$ are {\it congruent} over a group $G$ if there exists $Q \in G$ such that $Q^{-1}AQ = B$.

\begin{theorem}
\label{type1lemSo}
Suppose $\theta$ and $\phi$ are two Type 1 $k$-involutions of $\So(n,k,\beta)$ where $\theta = \Inn_A$ and $\phi = \Inn_B$. Then, $A = X \left(\begin{smallmatrix}-I_{m_A} & 0 \\0 & I_{n-m_A}\end{smallmatrix}\right) X^{-1}$ and  $B = Y\left(\begin{smallmatrix}-I_{m_B} & 0 \\0 & I_{n-m_B}\end{smallmatrix}\right) Y^{-1}$ where $m_A, m_B \le \frac{n}{2}$, and $$X  = \left(\begin{matrix}x_1 & x_2 & \cdots & x_n \end{matrix}\right), Y  = \left(\begin{matrix}y_1 & y_2 & \cdots & y_n \end{matrix}\right)\in \Gl(n,k)$$ 
have columns that are orthogonal eigenvectors of $A$ and $B$ respectively. We also have the diagonal matrices $$X^TMX = \left(\begin{array}{cc}X_1 & 0 \\0 & X_2\end{array}\right)$$ and $$Y^TMY = \left(\begin{array}{cc}Y_1 & 0 \\0 & Y_2\end{array}\right).$$ 
The following are equivalent:
 \begin{enumerate}
 
 \item $\theta$ is isomorphic to $\phi$ over $\So(n,k,\beta)$.
 
  \item $A$ is conjugate to $B$ or $-B$ over $\So(n,k,\beta)$.
  
  \item $ X_1$ is congruent to $Y_1$ over $\Gl(m,k)$ and $X_2$ is congruent to $Y_2$ over $\Gl(n-m,k)$, or $ X_1$ is congruent to $Y_2$ over $\Gl(\frac{n}{2},k)$ and $X_2$ is congruent to $Y_1$ over $\Gl(\frac{n}{2},k)$.
  
   \item $\theta$ is isomorphic to $\phi$ over $\oo(n,k,\beta)$.
   
     \item $A$ is conjugate to $B$ or $-B$ over $\oo(n,k,\beta)$.

 \end{enumerate}
\end{theorem}

\begin{proof}
The equivalence of $(i)$ and $(ii)$ follows from Lemma \ref{TidyLem}, as does the equivalence of $(iv)$ and $(v)$.

Next we show that $(ii)$ implies $(iii)$. First suppose that $Q^{-1}AQ = B$ for some $Q \in \So(n,k,\beta)$.  $Q^{-1}AQ = B$ implies $$Q^{-1} X \left(\begin{array}{cc}-I_{m_A} & 0 \\0 & I_{n-m_A}\end{array}\right)  X^{-1}Q = Y\left(\begin{array}{cc}-I_{m_B} & 0 \\0 & I_{n-m_B}\end{array}\right)  Y^{-1}.$$ Since the matrices on both sides of the equality above must have the same eigenvalues with the same multiplicities, then we see that $m_A = m_B$. Let $m = m_A = m_B$, and recall from Theorem \ref{sltheorem1} that $$-I_{m,n-m} = \left(\begin{array}{cc}-I_{m} & 0 \\0 & I_{n-m}\end{array}\right).$$ Rearranging the previous equation, we have $$I_{m,n-m}X^{-1}QY = X^{-1}QYI_{m,n-m},$$ which tells us that $X^{-1}QY = \left(\begin{smallmatrix}R_1 & 0 \\0 & R_2\end{smallmatrix}\right),$ where $R_1 \in \Gl(m,k)$ and $R_2\in \Gl(n-m,k)$. Rearranging, we have that $QY = X\left(\begin{smallmatrix}R_1 & 0 \\0 & R_2\end{smallmatrix}\right).$ Since $Q \in \oo(n,k,\beta)$, then we know that $Q^TMQ=M$.

So, 
\begin{align*}
Y^TMY &= Y^TQ^TMQY\\
&= \left(\begin{array}{cc}R_1 & 0 \\0 & R_2\end{array}\right)^T(X^TMX)\left(\begin{array}{cc}R_1 & 0 \\0 & R_2\end{array}\right).
\end{align*}
From here we see that 
$Y_1= R_1^TX_1R_1$ and 
$Y_2 = R_2^TX_2R_2.$

Now suppose that $Q^{-1}AQ = -B$ for some $Q \in \So(n,k,\beta)$.  This implies $$Q^{-1} X \left(\begin{array}{cc}-I_{m_A} & 0 \\0 & I_{n-m_A}\end{array}\right)  X^{-1}Q = Y\left(\begin{array}{cc}I_{m_B} & 0 \\0 & -I_{n-m_B}\end{array}\right)  Y^{-1}.$$ Since the matrices on both sides of the equality above must have the same eigenvalues with the same multiplicities, then we see that $m_A = n-m_B$. Since $m_A, m_B \le \frac{n}{2}$, then it follows that $m_A = m_B = \frac{n}{2}$. Rearranging the previous equation, we have $$I_{\frac{n}{2},\frac{n}{2}}X^{-1}QY = X^{-1}QYI_{\frac{n}{2},\frac{n}{2}},$$ which tells us that $X^{-1}QY = \left(\begin{smallmatrix} 0 &R_1 \\ R_2 & 0 \end{smallmatrix}\right),$ where $R_1, R_2 \in \Gl(\frac{n}{2},k)$. Rearranging, we have that $QY = X\left(\begin{smallmatrix} 0 &R_1 \\ R_2 & 0 \end{smallmatrix}\right).$ Since $Q \in \oo(n,k,\beta)$, then we know that $Q^TMQ=M$.

So, 
\begin{align*}
Y^TMY &= Y^TQ^TMQY\\ 
&=\left(\begin{array}{cc} 0 &R_1 \\ R_2 & 0 \end{array}\right)^T(X^TMX)\left(\begin{array}{cc} 0 &R_1 \\ R_2 & 0 \end{array}\right).
\end{align*}
From here we see that 
$Y_2= R_1^TX_1R_1$ and 
$Y_1 = R_2^TX_2R_2.$

This shows that $(ii)$ implies $(iii)$.

Now we show that $(iii)$ implies $(ii)$. Assume that $(iii)$ is the case. Specifically, assume that  $R_1 \in \Gl(m,k)$ and $R_2 \in \Gl(n-m,k)$ such that $Y_1= R_1^TX_1R_1$ and 
$Y_2 = R_2^TX_2R_2.$ Let $R = \left(\begin{smallmatrix}R_1 & 0 \\0 & R_2\end{smallmatrix}\right).$ So, we have $Y^TMY = R^T(X^TMX)R$. Let $Q = XRY^{-1}$. We will now show that $Q^{-1}AQ = B$ and that $Q \in \So(n,k,\beta)$.
\begin{align*}
Q^{-1}AQ &= (XRY^{-1})^{-1}A(XRY^{-1})\\
&= YR^{-1}X^{-1}AXRY^{-1}\\
&= YR^{-1}(-I_{m,n-m})RY^{-1}\\
&=  Y(-I_{m,n-m})Y^{-1} = B.
\end{align*}

Next, we must show that $Q \in \So(n,k,\beta)$. We first show that $Q^TMQ = M$. Recall that $Y^TMY = R^T(X^TMX)R$. So, 
\begin{align*}
Q^TMQ &= (XRY^{-1})^{T}M(XRY^{-1})\\ 
&= (Y^{-1})^T(R^TX^TMXR)Y^{-1}\\ 
&= (Y^{-1})^T(Y^TMY)Y^{-1}\\ 
&= M.
\end{align*}
In the event that $\det(Q) = -1$, then we can replace the first column of $X$ with its negative. This will have no effect on $R$ or $Y$, so the new $Q = XRY^{-1}$ have determinant 1, and it will still be the case that $Q^{-1}AQ = B$ and $Q^TMQ = M$. So, $Q \in \So(n,k,\beta)$.

If instead we assume that  $R_1 \in \Gl(m,k)$ and $R_2 \in \Gl(n-m,k)$ such that $Y_2= R_1^TX_1R_1$ and 
$Y_1 = R_2^TX_2R_2,$ then if we let $R = \left(\begin{smallmatrix} 0 &R_1 \\ R_2 & 0 \end{smallmatrix}\right)$, then we can let $Q = XRY^{-1}$ and get that $Q^{-1}AQ = -B$ and $Q \in \So(n,k,\beta)$. This shows that $(iii)$ implies $(ii)$.

We now show that $(iv)$ and $(v)$ are equivalent to the previous three conditions. First, we note that it is clear that $(i)$ implies $(iv)$. So, we need only show that $(iv)$ or $(v)$ implies one of the other three conditions. But, $(v)$ implies $(iii)$ from an argument very similar to the argument where we showed that $(ii)$ implies $(iii)$. Thus, all the conditions are equivalent.
\end{proof}

We note that this Theorem shows that isomorphy over $\So(n,k,\beta)$ and $\oo(n,k,\beta)$ are the same for Type 1 $k$-involutions. We will show in an explicit example that this does not occur in the Type 2 case. For the remaining three types of $k$-involutions, we will only find conditions for isomorphy over $\oo(n,k,\beta)$. Again, recall that these three Types of $k$-involutions only occur when $n$ is even. So, when $n$ is odd, we have isomorphy conditions over $\So(n,k,\beta)$. 

\subsection{Type 2 $k$-involutions}

We have a similar characterization of the matrices and isomorphy classes in the Type 2 case. We first prove a result about that characterizes the eigenvectors in the Type 2 case.

\begin{lem}

Suppose $A \in \oo(n,k[\sqrt{\alpha}],\beta) \setminus \oo(n,k,\beta) $ induces a Type 2 $k$-involution of  $\So(n,k,\beta)$ where $\sqrt{\alpha} \not \in k$. Recall that the entries of $A$ are all $k$-multiples of $\sqrt{\alpha}$. Suppose $x,y \in k^n$ such that $x+\sqrt{\alpha}y \in E(A,-1)$. Then, $x-\sqrt{\alpha}y \in E(A,1)$. Likewise, if $u,v \in k^n$ such that $u+\sqrt{\alpha}v \in E(A,1)$. Then, $u-\sqrt{\alpha}v \in E(A,-1)$. Further, $\dim(E(A,1))= \dim(E(A,-1))$.

\end{lem}

\begin{proof}

Suppose $a,b \in k$. Then, we refer to the Galois automorphism of the field $k[\sqrt{\alpha}] $ over $k$ that sends $a+\sqrt{\alpha}b $ to $a - \sqrt{\alpha}b$ as ``$\sqrt{\alpha}$-conjugation." Since it is an automorphism, it must preserve multiplication. Further, for a matrix $X = Y +\sqrt{\alpha}Z$ for $Y, Z$ with entries in $k$, we say that $Y -\sqrt{\alpha}Z$ is the ``$\sqrt{\alpha}$-conjugation" of $X$. We note that this mapping will preserve multiplication of matrices. 

Since $$A(x+\sqrt{\alpha}y) = -x-\sqrt{\alpha}y,$$ then we can take the ``$\sqrt{\alpha}$-conjugation" to see that $$(-A)(x-\sqrt{\alpha}) = -x+\sqrt{\alpha}y.$$ We note that the ``$\sqrt{\alpha}$-conjugation" of $A$ is $-A$ because each entry of $A$ is a $k$-multiple of $\sqrt{\alpha}$. We can multiply both side by -1 to see $$A(x-\sqrt{\alpha}) = x-\sqrt{\alpha}y.$$ That is, $x-\sqrt{\alpha}y \in E(A,1)$. This proves the first statement. An analogous argument proves the second.

To see that $\dim(E(A,1))= \dim(E(A,-1))$ is the case, note that the first statement tells us that $\dim(E(A,1)) \le \dim(E(A,-1))$, and that the second statement tells us that $\dim(E(A,1))\ge \dim(E(A,-1))$, since ``$\sqrt{\alpha}$-conjugation" is an invertible operator on $k[\sqrt{\alpha}]^n$.
\end{proof}

We are now able to characterize the Type 2 $k$-involutions.

\begin{lem}
\label{Type2ClassSo}
Suppose $\theta$ is a Type 2 $k$-involution of $\So(n,k,\beta)$. Let $A$ be the orthogonal matrix in $\oo(n,k[\sqrt{\alpha}],\beta)$ such that $\theta = \Inn_A$. Then, $$A = \frac{-\sqrt{\alpha}}{\alpha} X  \left(\begin{array}{cc}0 & I_{\frac{n}{2}} \\ \alpha I_{\frac{n}{2}} & 0\end{array}\right)  X^{-1}$$ where $$X  = \left(\begin{array}{cccccccc}x_1 & x_2 & \cdots & x_{\frac{n}{2}} &y_1 & y_2 & \cdots & y_{\frac{n}{2}} \end{array}\right)\in \Gl(n,k),$$ where 
 for each $i$, we have orthogonal vectors $x_i+\sqrt{\alpha}y_i \in E(A,-1)$ and orthogonal vectors $x_i-\sqrt{\alpha}y_i \in E(A,1)$. Further, $$X^TMX = \left(\begin{array}{cc}X_1 & X_2    \\X_2 & \frac{1}{\alpha}X_1   \end{array}\right)$$ where $X_1$ and $X_2$ are diagonal matrices.

\end{lem}

\begin{proof}

We begin by constructing bases for $E(A,1)$ and $E(A,-1)$ such that all the basis vectors lie in $k[\sqrt{\alpha}]^n$. From the previous lemma, we know that $\dim(E(A,1)) = \dim(E(A,-1)) =  \frac{n}{2}.$ (Note that this means that $n$ must be even for a Type 2 $k$-involution to occur.) Since $\Inn_A$ is a Type 1 $k$-involution of $\So(n,k[\sqrt{\alpha}], \beta)$, then we can apply Lemma \ref{Type1ClassSo} to find an orthogonal basis $\{ x_1+\sqrt{\alpha}y_1,...,x_{\frac{n}{2}}+\sqrt{\alpha}y_{\frac{n}{2}} \}$ of $E(A,-1)$, where $x_1,...,x_{\frac{n}{2}},y_1,...,y_{\frac{n}{2}} \in k^n$. By the previous lemma, we know that$\{x_1-\sqrt{\alpha}y_1,...,x_{\frac{n}{2}}-\sqrt{\alpha}y_{\frac{n}{2}}\}$ must be a basis for $E(A,1)$. Let $X  = \left(\begin{smallmatrix}x_1 & x_2 & \cdots & x_{\frac{n}{2}} &y_1 & y_2 & \cdots & y_{\frac{n}{2}} \end{smallmatrix}\right)\in \Gl(n,k).$

We now make a couple of observations. Suppose $u = x+\sqrt{\alpha} y$ is a -1-eigenvector of $A$ such that $x,y \in k^n$. Then, we know $v = x-\sqrt{\alpha} y$ is a 1-eigenvector of $A$. Observe that 
\begin{align*}
Ax &= \frac{1}{2}A(u+v)\\ 
&= \frac{1}{2}(-u+v)\\ 
&= -\sqrt{\alpha} y.
\end{align*} 
It follows from this that $$Ay = -\frac{\sqrt{\alpha}}{\alpha}x.$$ 

Since $Ax =  -\sqrt{\alpha} y$ and $Ay = -\frac{\sqrt{\alpha}}{\alpha}x$, then it follows that $$X^{-1}AX = \left(\begin{smallmatrix}0 & -\frac{\sqrt{\alpha}}{\alpha}I_{\frac{n}{2}} \\ -\sqrt{\alpha}I_{\frac{n}{2}} & 0\end{smallmatrix}\right).$$ Rearranging this, we see that $$A = -\frac{\sqrt{\alpha}}{\alpha} X  \left(\begin{array}{cc}0 & I_{\frac{n}{2}} \\ \alpha I_{\frac{n}{2}} & 0\end{array}\right)  X^{-1}.$$

Now, we need only prove the last statement to prove the Lemma. Since $\{ x_1+\sqrt{\alpha}y_1,...,x_{\frac{n}{2}}+\sqrt{\alpha}y_{\frac{n}{2}} \}$  is an orthogonal set of vectors, then we know when $i \ne j$ that 
\begin{align*}
0 &= \beta(x_i+\sqrt{\alpha}y_i, x_j+\sqrt{\alpha}y_j)\\ 
&= (\beta(x_i, x_j)+\alpha\beta(y_i, y_j))+\sqrt{\alpha}(\beta(x_i,y_j)+\beta(x_j,y_i)).
\end{align*}
This tells us that $$\beta(x_i, x_j)= -\alpha\beta(y_i, y_j)$$ and $$\beta(x_i,y_j)= -\beta(x_j,y_i).$$

Since vectors from $E(A,1)$ and $E(A,-1)$ are orthogonal, then we also know that 
\begin{align*}
0 &= \beta(x_i+\sqrt{\alpha}y_i, x_j-\sqrt{\alpha}y_j)\\ 
&= (\beta(x_i, x_j)-\alpha\beta(y_i, y_j))+\sqrt{\alpha}(-\beta(x_i,y_j)+\beta(x_j,y_i)),
\end{align*}
regardless of if $i$ and $j$ are distinct or equal.

This tells us that $$\beta(x_i, x_j)= \alpha\beta(y_i, y_j)$$ and $$\beta(x_i,y_j)= \beta(x_j,y_i).$$

So, when $i \ne j$, then we know that $$ \beta(x_i, y_j) = 0,$$
$$ \beta(x_i, x_j) = 0,$$

and
$$ \beta(y_i, y_j) = 0.$$

When $i = j$, we note that $$\beta(x_i, x_i)= \alpha\beta(y_i, y_i).$$  Then, we have $$X^TMX =  \left(\begin{array}{cc}X_1 & X_2    \\X_2 & \frac{1}{\alpha}X_1   \end{array}\right)$$ where $X_1$ and $X_2$ have been shown to be diagonal.
\end{proof}

We now show an example of a Type 2 $k$-involution, and apply the previous lemma to it.

\begin{beisp}
Assume that $\beta$ is the standard dot product. Then, $\Inn_A$ can be a Type 2 $k$-involution of $\So(4,\mathbb{Q})$ if $A$ is symmetric and orthogonal, since this will imply that $A^2 =I$, and if the entries of $A$ are all $k$-multiples of some $\sqrt{\alpha}$ such that $\sqrt{\alpha} \not \in k$ but $\alpha \in k$. Observe that the matrix $$A = \frac{\sqrt{3}}{3} \left(\begin{array}{cccc}0 & 1 & -1 & 1 \\1 & 0 & 1 & 1 \\-1 & 1 & 1 & 0 \\1 & 1 & 0 & -1\end{array}\right)$$ is both symmetric and orthogonal. Since each entry is the $\mathbb{Q}$-multiple of $\sqrt{3}$, then it is clear that $\Inn_A$ is a Type 2 $k$-involution of $\So(4, \mathbb{Q})$. It can be shown that $E(A,-1)$ has dimension 2. An orthogonal basis for this subspace is formed by the vectors $$v_1 =  \left(\begin{array}{c}\frac{1}{2} \\\frac{1}{2} \\0 \\1\end{array}\right)+ \sqrt{3} \left(\begin{array}{c}-\frac{1}{2} \\-\frac{1}{2} \\0 \\0\end{array}\right)$$ and $$ v_2=\left(\begin{array}{c}\frac{1}{2} \\ -\frac{1}{2} \\1 \\0\end{array}\right)+ \sqrt{3} \left(\begin{array}{c}\frac{1}{2} \\ -\frac{1}{2} \\0 \\0\end{array}\right).$$ 

It can be shown that $$v_3 =  \left(\begin{array}{c}\frac{1}{2} \\\frac{1}{2} \\0 \\1\end{array}\right)- \sqrt{3} \left(\begin{array}{c}-\frac{1}{2} \\-\frac{1}{2} \\0 \\0\end{array}\right)$$ and $$ v_4=\left(\begin{array}{c}\frac{1}{2} \\ -\frac{1}{2} \\1 \\0\end{array}\right)- \sqrt{3} \left(\begin{array}{c}\frac{1}{2} \\ -\frac{1}{2} \\0 \\0\end{array}\right)  $$ are orthogonal $1$-eigenvectors of $A$, where these are the $\sqrt{3}$-conjugates of $v_1$ and $v_2$, respectively.  

Following the notation of the previous lemma, we have $$X = \left(\begin{array}{cccc}\frac{1}{2} & \frac{1}{2} & -\frac{1}{2} & \frac{1}{2} \\\frac{1}{2} & -\frac{1}{2} & -\frac{1}{2} & -\frac{1}{2} \\0 & 1 & 0 & 0 \\1 & 0 & 0 & 0\end{array}\right),$$ where $X^TX = \left(\begin{smallmatrix}\frac{3}{3} & 0 & -\frac{1}{2} & 0 \\0 & \frac{3}{3} & 0 & \frac{1}{2} \\-\frac{1}{2} & 0 & \frac{1}{2} & 0 \\0 & \frac{1}{2} & 0 & \frac{1}{2}\end{smallmatrix}\right)$ and $A = -\frac{\sqrt{3}}{3}X\left(\begin{smallmatrix}0 & I_{\frac{n}{2}} \\ 3I_{\frac{n}{2}} & 0\end{smallmatrix}\right) X^{-1}$ .
\end{beisp}

We now find conditions in the Type 2 case that are equivalent to isomorphy.

\begin{theorem}
\label{type2lemSo}
Suppose $\theta$ and $\phi$ are two Type 2 $k$-involutions of $\So(n,k,\beta)$ where $\theta = \Inn_A$ and $\phi = \Inn_B$. Then,  $$A = -\frac{\sqrt{\alpha}}{\alpha} X  \left(\begin{array}{cc}0 & I_{\frac{n}{2}} \\ \alpha I_{\frac{n}{2}} & 0\end{array}\right)  X^{-1} \in \oo(n,k[\sqrt{\alpha}],\beta)$$ where $$X  = \left(\begin{array}{cccccccc}x_1 & x_2 & \cdots & x_{\frac{n}{2}} & y_1 & y_2 & \cdots  & y_{\frac{n}{2}}  \end{array}\right)\in \Gl(n,k)$$ and the $x_i +\sqrt{\alpha}y_i$ are the orthogonal basis of $E(A,-1)$, and $$X^TMX = \left(\begin{array}{cc}X_1 & X_2    \\X_2 & \frac{1}{\alpha}X_1   \end{array}\right)$$ where $X_1$ and $X_2$ are diagonal matrices,

 and   $$B = -\frac{\sqrt{\gamma}}{\gamma} Y  \left(\begin{array}{cc}0 & I_{\frac{n}{2}} \\ \gamma I_{\frac{n}{2}} & 0\end{array}\right)  Y^{-1}\in \oo(n,k[\sqrt{\gamma}],\beta)$$ where $$Y  = \left(\begin{array}{cccccccc}\tilde{x}_1 & \tilde{x}_2 & \cdots & \tilde{x}_{\frac{n}{2}} & \tilde{y}_1 & \tilde{v}_2 & \cdots  & \tilde{y}_{\frac{n}{2}}  \end{array}\right)\in \Gl(n,k)$$ and the $\tilde{x}_i +\sqrt{\gamma}\tilde{y}_i$ is the orthogonal eigenvectors of $E(B,-1)$, and $$Y^TMY =  \left(\begin{array}{cc}Y_1 & Y_2    \\Y_2 & \frac{1}{\gamma}Y_1   \end{array}\right)$$ where $Y_1$ and $Y_2$ are diagonal matrices,
 and the following are equivalent:
 
 \begin{enumerate}
 \item $\theta$ is isomorphic to $\phi$ over $\oo(n,k,\beta)$.
 
  \item $A$ is conjugate to $B$ or $-B$ over $\oo(n,k,\beta)$.
  
  \item $\alpha = \gamma$ and $Y^TMY = R^TX^TMXR$ where $R = \left(\begin{smallmatrix}R_1 & R_2 \\ \alpha R_2 & R_1\end{smallmatrix}\right) \in \Gl(n,k)$, or $\alpha = \gamma$ and $Y^TMY = R^TX^TMXR$ where $R = \left(\begin{smallmatrix}R_1 & R_2 \\ -\alpha R_2 & -R_1\end{smallmatrix}\right) \in \Gl(n,k)$.
  
  \item We can choose $X$ and $Y$ such that $\alpha = \gamma$, and for $R = \left(\begin{smallmatrix}R_1 & R_2 \\ \alpha R_2 & R_1\end{smallmatrix}\right) \in \Gl(n,k)$ we have 
    $$Y_1 = R_1^TX_1R_1 +\alpha R_2^TX_2R_1+\alpha R_1^TX_2R_2+\alpha R_2^TX_1R_2$$ 
  and
  $$Y_2 = R_2^TX_1R_1+R_1^TX_2R_1+\alpha R_2^TX_2R_2+R_1^TX_1R_2,$$
  
  or for $R = \left(\begin{smallmatrix}R_1 & R_2 \\ -\alpha R_2 & -R_1\end{smallmatrix}\right) \in \Gl(n,k)$ we have
 $$Y_1 = R_1^TX_1R_1 -\alpha R_2^TX_2R_1-\alpha R_1^TX_2R_2+\alpha R_2^TX_1R_2$$ 
  and
  $$Y_2 = R_2^TX_1R_1-R_1^TX_2R_1-\alpha R_2^TX_2R_2+R_1^TX_1R_2.$$
 
 \end{enumerate}
\end{theorem}

\begin{proof}

The equivalence of $(i)$ and $(ii)$ follows from Lemma \ref{TidyLem}. So, we begin by showing that $(ii)$ implies $(iii)$. First suppose there exists $Q \in \oo(n,k,\beta)$ such that $Q^{-1}AQ = B$. So, we have $$Q^{-1} \frac{\sqrt{\alpha}}{\alpha} X  \left(\begin{array}{cc}0 & I_{\frac{n}{2}} \\ \alpha I_{\frac{n}{2}} & 0\end{array}\right)  X^{-1} Q = \frac{\sqrt{\gamma}}{\gamma} Y  \left(\begin{array}{cc}0 & I_{\frac{n}{2}} \\ \gamma I_{\frac{n}{2}} & 0\end{array}\right)  Y^{-1}.$$ Also, we know that since $A \in \So(n,k[\sqrt{\alpha}],\beta)$ and $B\in \So(n,k[\sqrt{\gamma}],\beta)$ are congruent over $\oo(n,k,\beta)$, then $\gamma$ must be a $k$-multiple of $\alpha$. Without loss of generality, we will assume $\gamma = \alpha$. Thus, $$Q^{-1}  X  \left(\begin{array}{cc}0 & I_{\frac{n}{2}} \\ \alpha I_{\frac{n}{2}} & 0\end{array}\right)  X^{-1} Q = Y  \left(\begin{array}{cc}0 & I_{\frac{n}{2}} \\ \alpha I_{\frac{n}{2}} & 0\end{array}\right)  Y^{-1}.$$ 

Rearranging, we see that $$\left(\begin{array}{cc}0 & I_{\frac{n}{2}} \\ \alpha I_{\frac{n}{2}} & 0\end{array}\right)X^{-1}QY  = X^{-1}QY \left(\begin{array}{cc}0 & I_{\frac{n}{2}} \\ \alpha I_{\frac{n}{2}} & 0\end{array}\right).$$ Let $R = X^{-1}QY $, and note that $R \in \Gl(n,k)$. Since $\left(\begin{smallmatrix}0 & I_{\frac{n}{2}} \\ \alpha I_{\frac{n}{2}} & 0\end{smallmatrix}\right) R = R \left(\begin{smallmatrix}0 & I_{\frac{n}{2}} \\ \alpha I_{\frac{n}{2}} & 0\end{smallmatrix}\right)$, then $R = \left(\begin{smallmatrix}R_1 & R_2 \\ \alpha R_2 & R_1\end{smallmatrix}\right)$. Observe that $XR = QY$. Also, observe that since $Q \in \So(n,k,\beta)$, then we know that $Q^TMQ=M$. It follows from these observations that 
\begin{align*} 
R^T(X^TMX)R &= (XR)^TM(XR)\\ 
&= (QY)^TM(QY)\\
&= Y^T(Q^TMQ)Y\\ 
&= Y^TMY.
\end{align*} 

If instead we assume that there exists $Q \in \oo(n,k,\beta)$ such that $Q^{-1}AQ = -B$, then we can similarly show that  $\alpha = \gamma$ and $Y^TMY = R^TX^TMXR$ where $R = \left(\begin{smallmatrix}R_1 & R_2 \\ -\alpha R_2 & -R_1\end{smallmatrix}\right) \in \Gl(n,k)$ for $R_1, R_2 \in \M(\frac{n}{2},k).$ This proves that $(ii)$ implies $(iii)$. 

We now show that $(iii)$ implies $(ii)$. First assume  $\alpha = \gamma$ and $X^TMX$ is congruent to $Y^TMY$ over $\Gl(n,k)$ where $Y^TMY = R^TX^TMXR$ for $R = \left(\begin{smallmatrix}R_1 & R_2 \\ \alpha R_2 & R_1\end{smallmatrix}\right)$, where $R_1, R_2 \in \Gl(\frac{n}{2},k)$. Let $Q = XRY^{-1}.$ Then, we observe that 
\begin{align*}
Q^{-1}AQ &= (XRY^{-1})^{-1}A(XRY^{-1})\\
&= YR^{-1}(X^{-1}AX)RY^{-1}\\
&= -\frac{-\sqrt{\alpha}}{\alpha} YR^{-1} \left(\begin{array}{cc}0 & I_{\frac{n}{2}} \\ \alpha I_{\frac{n}{2}} & 0\end{array}\right) RY^{-1}\\ 
&=  -\frac{-\sqrt{\alpha}}{\alpha} Y \left(\begin{array}{cc}0 & I_{\frac{n}{2}} \\ \alpha I_{\frac{n}{2}} & 0\end{array}\right) R^{-1}RY^{-1}\\
&= -\frac{-\sqrt{\alpha}}{\alpha} Y\left(\begin{array}{cc}0 & I_{\frac{n}{2}} \\ \alpha I_{\frac{n}{2}} & 0\end{array}\right) Y^{-1} = B.
\end{align*}

To show that $(ii)$ is indeed the case, we need only show that $Q \in \oo(n,k,\beta)$. By construction, we know that $Q \in \Gl(n,k)$.  So, it is suffice to show $Q^TMQ = M$. But, 
\begin{align*}
Q^TMQ &= (XRY^{-1})^TM(XRY^{-1})\\
&= (Y^{-1})^T(R^TX^TMXR)Y^{-1}\\ 
&= (Y^{-1})^T(Y^TMY)Y^{-1} = M.
\end{align*}
If we instead assume that $\alpha = \gamma$ and $X^TMX$ is congruent to $Y^TMY$ over $\Gl(n,k)$ where $Y^TMY = R^TX^TMXR$ for $R = \left(\begin{smallmatrix}R_1 & R_2 \\ -\alpha R_2 & -R_1\end{smallmatrix}\right)$, where $R_1, R_2 \in \Gl(\frac{n}{2},k)$, then if we let $Q = XRY^{-1},$ we can similarly show that $Q^{-1}AQ = -B$ and $Q \in \oo(n,k,\beta)$. This shows that $(iii)$ implies $(ii)$.

Lastly, matrix multiplication shows that $(iii)$ and $(iv)$ are equivalent. 

 \end{proof}
 
 The reader will notice that in the Type 1 case, our conditions gave us isomorphy of $k$-involutions over $\So(n,k,\beta)$, but the Type 2 case gave us isomorphy of $k$-involutions over $\oo(n,k,\beta)$. In the following example, we give an example that shows that isomorphy over $\oo(n,k,\beta)$ is not the same as isomorphy over $\So(n,k,\beta)$ for Type 2 $k$-involutions.
 
 \begin{beisp}
 
 Consider the group $\So(4, \mathbb{F}_3)$. That is, consider the case where $k$ is the group of three elements, and the bilinear form is the standard dot product. Let $i$ denote a fixed square root of $2 =-1$. A Type 2 $k$-involution is induced by the matrix 
 $$A = i \left(\begin{array}{cccc}1 & 1 & 0 & 0 \\1 & 2 & 0 & 0 \\0 & 0 & 1 & 1 \\0 & 0 & 1 & 2\end{array}\right) \in \So(4, \mathbb{F}_3[i]).$$ 
 By analyzing suitable eigenvectors for this matrix, we see that 
 $$X = \left(\begin{array}{cccc}1 & 0 & 2 & 0 \\0 & 0 & 2 & 0 \\0 & 1 & 0 & 2 \\0 & 0 & 0 & 2\end{array}\right)$$ 
 where 
 $$A = i X^{-1}  \left(\begin{array}{cccc}0 & 0 & 1 & 0 \\0 & 0 & 0 & 1 \\2 & 0 & 0 & 0 \\0 & 2 & 0 & 0\end{array}\right)  X.$$ 
We also see that
 $$X^TX = \left(\begin{array}{cc}X_1 & X_2    \\X_2 & 2X_1   \end{array}\right)$$ where $X_1 = \left(\begin{smallmatrix}1 & 0 \\0 & 1\end{smallmatrix}\right)$ and $X_2 = \left(\begin{smallmatrix}2 & 0 \\0 & 2\end{smallmatrix}\right).$

Now, we also consider the Type 2 $k$-involution of $\So(4, \mathbb{F}_3)$ that is induced by the matrix 
$$B = i \left(\begin{array}{cccc}0 & 0 & 2 & 1 \\0 & 0 & 2 & 2 \\2 & 2 & 0 & 0 \\1 & 2 & 0 & 0\end{array}\right) \in \So(4, \mathbb{F}_3[i]).$$ 
 By analyzing suitable eigenvectors for this matrix, we see that for 
 $$Y = \left(\begin{array}{cccc}1 & 0 & 0 & 0 \\0 & 1 & 0 & 0 \\0 & 0 & 1 & 1 \\0 & 0 & 2 & 1\end{array}\right),$$ 
 we have 
 $$B = i Y^{-1}  \left(\begin{array}{cccc}0 & 0 & 1 & 0 \\0 & 0 & 0 & 1 \\2 & 0 & 0 & 0 \\0 & 2 & 0 & 0\end{array}\right)  Y.$$ 
We also see that 
 $$Y^TY = \left(\begin{array}{cc}Y_1 & Y_2    \\Y_2 & 2Y_1   \end{array}\right)$$ where $Y_1 = \left(\begin{smallmatrix}1 & 0 \\0 & 1\end{smallmatrix}\right)$ and $Y_2 = \left(\begin{smallmatrix}0 & 0 \\0 & 0\end{smallmatrix}\right).$

 We have two ways of showing that these Type 2 $k$-involutions are congruent over $\oo(4,\mathbb{F}_3)$. First, we consider the matrix 
 $$Q = \left(\begin{array}{cccc}1 & 2 & 2 & 1 \\1 & 1 & 1 & 1 \\1 & 1 & 2 & 2 \\2 & 1 & 2 & 1\end{array}\right) \in \oo(4, \mathbb{F}_3) \setminus \So(4, \mathbb{F}_3).$$ 
Then, $B = Q^{-1}AQ$. This is condition $(ii)$ of the previous theorem.
 
 Secondly, if we let $R = \left(\begin{smallmatrix}R_1 & R_2 \\ 2R_2 & R_1\end{smallmatrix}\right) \in \Gl(4,\mathbb{F}_3)$ where $R_1 = \left(\begin{smallmatrix}1 & 0 \\0 & 1\end{smallmatrix}\right)$ and $R_2 = \left(\begin{smallmatrix}1 & 1 \\2 & 1\end{smallmatrix}\right),$ then we get that $R^TY^TYR= X^TX.$ This is condition $(iii)$ of the previous theorem. 
 
%
 We now show that there does not exist $W \in \So(4, \mathbb{F}_3)$ such that $B = W^{-1}AW$. We proceed by contradiction and suppose that these does exist such an $W$. It then follows that $A$ and $QW^{-1} \in \oo(4, \mathbb{F}_3)$ are commuting matrices. It is a simple matter to show that matrices that commute with $A$ must be of the form 
 $$\left(\begin{array}{cccc}a & b & c & d \\b & a+b & d & c+d \\e & f & g & h \\f & e+f & h & g+h\end{array}\right).$$
 One such matrix is 
 $$\left(\begin{array}{cccc}1 & 1 & 2 & 2 \\1 & 2 & 2 & 1 \\2 & 2 & 2 & 2 \\2 & 1 & 2 & 1\end{array}\right) \in \So(4, \mathbb{F}_3).$$
 But, all other such orthogonal matrices differ from this matrix only in that an even number of rows and/or columns have been multiplied by $2=-1$ or an even number of rows and/or columns have been swapped. All of these actions create matrices that will also have determinant 1. Thus, all the matrices in $\oo(4, \mathbb{F}_3)$ which commute with $A$ are also members of $\So(4, \mathbb{F}_3),$ which contradicts $QW^{-1} \in \oo(4, \mathbb{F}_3)$ commuting with $A$. So, no such $W \in \So(4,\mathbb{F}_3)$ can exist, which means we cannot strengthen the above theorem by replacing $\oo(n,k,\beta)$ with $\So(n,k,\beta)$ in conditions $(i)$ and $(ii)$.
 \end{beisp}
 
 \begin{theorem}
 Suppose $\Inn_A$ and $\Inn_B$ are both Type 2 $k$-involutions of $\So(n,k,\beta)$. Then, $\Inn_A$ and $\Inn_B$ are isomorphic over $\oo(n,k,\beta)$ if and only if they are isomorphic over \newline $\So(n,k[\sqrt{\alpha}], \beta)$.
 \end{theorem}
 
 \begin{proof}
 
When viewed as Type 2 $k$-involutions of $\So(n,k,\beta)$, we can write  $$A = -\frac{\sqrt{\alpha}}{\alpha} U  \left(\begin{array}{cc}0 & I_{\frac{n}{2}} \\ \alpha I_{\frac{n}{2}} & 0\end{array}\right)  U^{-1} \text{ where } U^TMU = \left(\begin{array}{cc}U_1 & U_2    \\U_2 & \frac{1}{\alpha}U_1   \end{array}\right)$$ and $$B = -\frac{\sqrt{\alpha}}{\alpha} V  \left(\begin{array}{cc}0 & I_{\frac{n}{2}} \\ \alpha I_{\frac{n}{2}} & 0\end{array}\right)  V^{-1} \text{ where } V^TMV =  \left(\begin{array}{cc}V_1 & V_2    \\V_2 & \frac{1}{\alpha}V_1   \end{array}\right)$$ and $U_1$, $U_2$, $V_1$ and $V_2$ are diagonal matrices.

 When $\Inn_A$ and $\Inn_B$ are viewed as $k$-involutions of $\So(n,k[\sqrt{\alpha}], \beta)$, then they are Type 1 $k$-involutions. Further, we can choose $X$ and $Y\in \Gl(n,k[\sqrt{\alpha}])$ such that $A = X \left(\begin{smallmatrix}-I_{\frac{n}{2}} & 0 \\0 & I_{\frac{n}{2}}\end{smallmatrix}\right)X^{-1}$,  $B = Y \left(\begin{smallmatrix}-I_{\frac{n}{2}} & 0 \\0 & I_{\frac{n}{2}}\end{smallmatrix}\right)Y^{-1}$, and 
 
$$X_1 = \frac{1}{2}(U_1 +\sqrt{\alpha}U_2),$$
$$X_2 =  \frac{1}{2}(U_1 -\sqrt{\alpha}U_2),$$
$$Y_1 = \frac{1}{2}(V_1 +\sqrt{\alpha}V_2),$$
 and $$Y_2 = \frac{1}{2}(V_1 -\sqrt{\alpha}V_2).$$
 
 This follows from the way in which $U$ and $V$ are constructed from the eigenvalues of $A$ and $B$. We need to simply have $X$ and $Y$ consist of the appropriate eigenvectors, and mandate that the last $\frac{n}{2}$ columns of $X$ and $Y$ are the $\sqrt{\alpha}$-conjugates of the first $\frac{n}{2}$ columns. (The only exception to this is that we may need to negate the first column of $X$, so that we can preserve isomorphy of $\Inn_A$ and $\Inn_B$ over $\So(n,k[\sqrt{\alpha}],\beta)$, if we are assuming that. But, this will not change the value of $X_1$.)

Now, suppose $\Inn_A$ and $\Inn_B$ are isomorphic over $\So(n,k[\sqrt{\alpha}],\beta)$ as Type 1 $k$-involutions. Then, from Theorem \ref{type1lemSo} we know that $Y_1$ is congruent to either $X_1$ or $X_2$. 

  In the first case, we see that 
\begin{align*}
\frac{1}{2}(V_1 +\sqrt{\alpha}V_2) &= Y_1\\
&=  (R_1+\sqrt{\alpha}R_2)^T X_1(R_1+\sqrt{\alpha}R_2)\\
&= (R_1+\sqrt{\alpha}R_2)^T \frac{1}{2}(U_1 +\sqrt{\alpha}U_2)(R_1+\sqrt{\alpha}R_2),
\end{align*}
  where $R_1$ and $R_2$ are over $k$.

In the second case, we see that 
\begin{align*}
\frac{1}{2}(V_1 +\sqrt{\alpha}V_2) &= Y_1\\
&=  (R_1+\sqrt{\alpha}R_2)^T X_2(R_1+\sqrt{\alpha}R_2)\\
&= (R_1+\sqrt{\alpha}R_2)^T \frac{1}{2}(U_1 -\sqrt{\alpha}U_2)(R_1+\sqrt{\alpha}R_2),
\end{align*}
where $R_1$ and $R_2$ are over $k$.

It follows from this that 

$$V_1 = R_1^TU_1R_1 +\alpha R_2^TU_2R_1+\alpha R_1^TU_2R_2+\alpha R_2^TU_1R_2$$ 
  and
  $$V_2 = R_2^TU_1R_1+R_1^TU_2R_1+\alpha R_2^TU_2R_2+R_1^TU_1R_2,$$
  or
 $$V_1 = R_1^TU_1R_1 -\alpha R_2^TU_2R_1-\alpha R_1^TU_2R_2+\alpha R_2^TU_1R_2,$$ 
  and
  $$V_2 = R_2^TU_1R_1-R_1^TU_2R_1-\alpha R_2^TU_2R_2+R_1^TU_1R_2.$$

The previous theorem tells us that this means that $\Inn_A$ and $\Inn_B$ are isomorphic over $\oo(n,k,\beta)$. Since the converse is clear, then we have shown what was needed.
 
 \end{proof}

\subsection{Type 3 $k$-involutions}

We now examine the Type 3 case. Recall that $\phi$ is a Type 3 $k$-involution if $\phi = \Inn_A$, where $A \in \oo(n,k,\beta)$ and $A^2 = -I$. Such matrices have eigenvalues $\pm i$, where $i$ is a fixed square root of $-1$, and are diagonalizable because the minimal polynomial has no repeated roots. We refer to the Galois automorphism that send $a+bi$ to $a-bi$ for $a, b \in k$ as complex conjugation. 

We begin by proving a couple or results about the eigenvectors of such matrices.

\begin{lem}

Suppose $A \in \oo(n,k,\beta) $ induces a Type 3 $k$-involution of $\So(n,k,\beta)$. Also suppose $x,y \in k^n$ such that $x+iy \in E(A,-i)$. Then, $x-iy \in E(A,i)$. Likewise, if $u,v \in k^n$ such that $u+iv \in E(A,i)$, then $u-iv \in E(A,-i)$. Further, $\dim(E(A,i))= \dim(E(A,-i))$.

\end{lem}

\begin{proof}
Suppose $x,y \in k^n$ such that $x+iy \in E(A,-i)$. Then, 
$$A(x+iy) = -i(x+iy)$$ implies
$$Ax+iAy = y-ix.$$ If we take the complex conjugate, then we see that 
$$Ax-iAy = y+ix.$$ This implies
$$A(x-iy) = i(x-iy),$$ which shows that $x-iy \in E(A,i)$. A similar proof will show that if $u,v \in k^n$ such that $u+iv \in E(A,i)$, then $u-iv \in E(A,-i)$.

Since  $x+iy \in E(A,-i)$ implies $x-iy \in E(A,i)$ and vice versa, then we see that $\dim(E(A,i))$ $= \dim(E(A,-i))$. 
\end{proof}

\begin{lem}
\label{Type3EigenSo}
Suppose $\theta = \Inn_A$ is a Type 3 $k$-involution of $\So(n,k,\beta)$ where $A \in \oo(n,k,\beta)$. Then, we can find $x_1,...,x_{\frac{n}{2}}, y_1,...,y_{\frac{n}{2}} \in k^n$ such that the $x_j+iy_j$ are a basis for $E(A,-i)$ and the $x_j-iy_j$ are a basis for $E(A,i)$.
\end{lem}

\begin{proof}

Since $\Inn_A$ is Type 3, then we are assuming that $A \in \oo(n,k, \beta)$ and $A^2 = -I$. Note that this also means that $n$ is even. It follows that all eigenvalues of $A$ are $\pm i$. Since there are no repeated roots in the minimal polynomial of $A$, then we see that $A$ is diagonalizable. We begin by constructing bases for $E(A,i)$ and $E(A,-i)$ such that all the basis vectors lie in $k[i]^n$.  Let $\{z_1,...,z_n\}$ be a basis for $k^n$. For each $j$, let $u_j = z_j+iAz_j$ Note that 
\begin{align*}
Au_j &= A(z_j+iAz_j)\\ 
&= (A+iA^2)z_j\\ 
&= (A-iI)z_j\\ 
&= -i(z_j+iAz_j)\\ 
&= -iu_j.
\end{align*} 
So, $\{u_1,...,u_n\}$ must span $E(A,-i)$. Thus, we can appropriately choose $\frac{n}{2}$ of these vectors and form a basis for $E(A,-i)$. Note that each of these vectors lies in $k[i]^n$. Label these basis vectors as $v_1,...,v_\frac{n}{2}$. We can write each of these vectors as $v_j = x_j+iy_j$. By the previous lemma, we know that $x_j-iy_j \in E(A,i)$. Since these vectors will be linearly independent, then they form a basis for $E(A,i)$.
\end{proof}

We are now able to prove results that characterize the matrices that induce Type 3 $k$-involutions, and then use these characterizations to find conditions on these $k$-involutions that are equivalent to isomorphy. We will have to prove our result by looking at separate cases, depending on whether or not $i$ lies in $k$. We begin by assuming that $i \in k$.

\begin{lem}
\label{Type3ClassYesSo}
Assume $i \in k$ and suppose $\theta = \Inn_A$ is a Type 3 $k$-involution of $\So(n,k,\beta)$, where $A \in \oo(n,k,\beta)$. Then, $A = X \left(\begin{smallmatrix}-iI_{\frac{n}{2}} &0 \\ 0& iI_{\frac{n}{2}}\end{smallmatrix}\right) X^{-1}$ for some $X \in \Gl(n,k),$ where $X^TMX = \left(\begin{smallmatrix}0 & X_1\\ X_1 & 0\end{smallmatrix}\right)$, where $X_1$ is a diagonal matrix.
\end{lem}

\begin{proof}

We know from Lemma \ref{Type3EigenSo} that we have bases for $E(A,-i)$ and $E(A,I)$ that lie in $k^n$. We will show that we can in fact choose bases $a_1,...,a_{\frac{n}{2}}$ for $E(A,-i) \cap k^n$ and $b_1,...,b_{\frac{n}{2}}$ for $E(A,i) \cap k^n$ such that $\beta(a_j,a_l) = 0 = \beta(b_j,b_l)$ and $\beta(a_j, b_l)$ is nonzero if and only if $j=l$. We will build these bases recursively.

First, we know that we can choose some nonzero $a_1 \in E(A,-i) \cap k^n$. Then, since $\beta$ is non degenerate, we can choose a vector $t$ such that $\beta(a_1, t) \ne 0$. We note that $E(A,-i) \oplus E(A,i) = k^n$, so we can choose $t_{-i} \in E(A,-i)\cap k^n$ and $t_i \in E(A,i)\cap k^n$ such that $t = t_{-i}+t_i$. Since $\beta(a_1, t_{-i}) = 0$, then it follows that $\beta(a_1, t_i) \in k$ is nonzero. Let $b_1 = t_i$.

Let $E_1 = \Span_k(a_1,b_1)$ and let $F_1$ be the orthogonal complement of $E_1$ in $k^n$. Since the system of linear equations $$\beta(a_1,x) = 0$$ $$\beta(b_1,x) =0$$ has $n-2$ free variables, then we see that $F_1$ has dimension $n-2$. 

We now need to find $a_2 \in F_1 \cap E(A,-i)$. Similar to the construction in the previous lemma, we can choose $x \in F_1$, and let $a_2 = x+iAx$. It follows that $a_2 \in F_1 \cap E(A,-i)$. Now we want $b_2 \in F_2 \cap E(A,i)$ such that $\beta(a_2, b_2) =1$. Since $\beta|_{F_1}$ is non degenerate, then there exists some $y \in F_2$ such that $\beta(a_2,y) \ne 0$. Similar to the construction of $b_1$, we see that this implies the existence a vector $b_2$ that fits our criteria. 

Now, we let $E_2 = \Span_k(a_1,a_2,b_1,b_2)$ and let $F_2$ be the orthogonal complement of $E_2$ in $k^n$. We continue this same argument $\frac{n}{2}$ times, until we have the bases that we wanted to find. Let $$X= (a_1,...,a_{\frac{n}{2}}, b_1,...,b_{\frac{n}{2}}).$$ Then, the result follows.
\end{proof}

\begin{theorem}
\label{type3lemYesSo}
Assume that $i \in k$. Then, if $\Inn_A$ and $\Inn_B$ are both Type 3 $k$-involutions of $\So(n,k,\beta)$, then $\Inn_A$ and $\Inn_B$ are isomorphic over $\oo(n,k,\beta)$.
\end{theorem}

\begin{proof}

Suppose we have two such $k$-involutions of $\So(n,k,\beta)$. Let them be represented by matrices $A,B \in \oo(n,k,\beta)$. By the previous Lemma, we can choose diagonal $X, Y \in \Gl(n,k)$ such that $$X^{-1}AX = \left(\begin{array}{cc}-iI & 0 \\0 & iI\end{array}\right) = Y^{-1}BY,$$ $$X^TMX = \left(\begin{array}{cc}0 & X_1\\ X_1 & 0\end{array}\right),$$ and  $$Y^TMY = \left(\begin{array}{cc}0 & Y_1\\ Y_1 & 0\end{array}\right).$$

Since $X_1$ and $Y_1$ are both invertible diagonal matrices, then we can choose $R_1$ and $R_2 \in \Gl(\frac{n}{2},k)$ such that $Y_1 = R_1^TX_1R_2$. Let $R = \left(\begin{smallmatrix}R_1 & 0 \\0 & R_2\end{smallmatrix}\right)$ and $Q = XRY^{-1}$. It follows from this that $R^TX^TMXR = Y^TMY$. We will show that $Q \in \oo(n,k,\beta)$ and $Q^{-1}AQ = B$. This will then prove that $\Inn_A$ and $\Inn_B$ lie in the same isomorphy class by Lemma \ref{TidyLem}. 

First we show that $Q \in \oo(n,k,\beta)$. Note that 
\begin{align*}
Q^TMQ &= (XRY^{-1})^TM(XRY^{-1})\\ 
&= (Y^{-1})^TR^T(X^TMX)RY^{-1}\\ 
&= (Y^{-1})^T(Y^TMY)Y^{-1}\\
&= M,
\end{align*} 
which proves this claim. 

Lastly, we show that $Q^{-1}AQ = B$. We first note that $R$ and $\left(\begin{smallmatrix}-iI & 0 \\0 & iI\end{smallmatrix}\right)$ commute. Then, we see that
\begin{align*}
Q^{-1}AQ &= (XRY^{-1})^{-1}A (XRY^{-1})\\
&= YR^{-1}(X^{-1}AX)RY^{-1}\\
&= Y R^{-1}\left(\begin{array}{cc}-iI & 0 \\0 & iI\end{array}\right)R Y^{-1}\\ 
&= Y R^{-1}R\left(\begin{array}{cc}-iI & 0 \\0 & iI\end{array}\right) Y^{-1}\\
&= Y \left(\begin{array}{cc}-iI & 0 \\0 & iI\end{array}\right) Y^{-1}\\
&= B.
\end{align*}
 
\end{proof}

We now begin examining the case where $i \not \in k$.

\begin{lem}
\label{Type3ClassNoSo}
Assume $i \not \in k$ and suppose $\theta=\Inn_A$ is a Type 3 $k$-involution of $\So(n,k,\beta)$. Then, $A = U \left(\begin{smallmatrix}0 & -I_{\frac{n}{2}} \\ I_{\frac{n}{2}} & 0\end{smallmatrix}\right) U^{-1}$ for $$U  = \left(\begin{array}{cccccccccc}a_1 & a_2 & \cdots & a_\frac{n}{2} &b_1 & b_2 & \cdots & b_\frac{n}{2}  \end{array}\right)\in \Gl(n,k),$$ where the $a_j+ib_j$ are a basis for $E(A,-i)$, the $a_j-ib_j$ are a basis for $E(A,i)$, and $U^TMU = \left(\begin{smallmatrix}U_1 & 0 \\ 0 & U_1\end{smallmatrix}\right)$ is a diagonal matrix.
\end{lem}

\begin{proof}

We know from Lemma \ref{Type3EigenSo} that we have bases for $E(A,-i)$ and $E(A,I)$ that lie in $k[i]^n$. We will show that we can in fact choose bases $a_1+ib_1,...,a_{\frac{n}{2}}+ib_{\frac{n}{2}}$ for $E(A,-i) \cap k[i]^n$ and $a_1-ib_1,...,a_{\frac{n}{2}}-ib_{\frac{n}{2}}$ for $E(A,i) \cap k[i]^n$ such that $\beta(a_j+ib_j, a_l-ib_l)$ is nonzero if and only if $j=l$. From this, we will be able to show that $\beta(a_j,a_l) = 0 = \beta(b_j,b_l)$ when $j \ne l$ and $\beta(a_j, b_l) = 0$ for all $j$ and $l$. We will build these bases recursively.

Recall that given any vector $x \in k^n$, we know that $x+iAx \in E(A,-i)$. We want to choose $x\in k^n$ such that $\beta(x,x) \ne 0$. (The reasons for this will become apparent.) $M$ is an invertible matrix, so there are at least $n$ instances of $e_j^TMe_l \ne 0$. If there is an instance where $j=l$, let $x=e_j$. If not, then instead we have  $e_j^TMe_l = 0 = e_l^TMe_j$, and we let $x = e_j+e_l$. Then, 
\begin{align*}
\beta(x,x) &= \beta(e_j+e_l, e_j+e_l)\\ 
&= 2 \beta(e_j,e_l)\\ 
&\ne 0.
\end{align*} 

So, we have $x \in k^n$ such that $\beta(x,x) \ne 0$, and we have $x+iAx \in E(A,-i)$. Let $a_1 = x$ and $b_1 = Ax$. So, $a_1+ib_1 \in E(A,-i)$ and $a_1-ib_1 \in E(A,i)$.  From this, it follows that 
\begin{align*}
\beta(a_1+ib_1, a_1-ib_1) &= (\beta(a_1,a_1) +\beta(b_1,b_1))+i(-\beta(a_1,b_1)+\beta(a_1,b_1)\\ 
&= 2\beta(a_1,a_1) = 2\beta(x,x)\\ 
&\ne 0.
\end{align*}

Let $E_1 = \Span_{k[i]}(a_1+ib_1, a_1-ib_1) = \Span_{k[i]}(a_1,b_1)$, and let $F_1$ be the orthogonal complement of $E_1$ over $k[i]$. $F_1$ has dimension $n-2$, and $\beta|_{F_1}$ is nondegenerate. So, we can find a nonzero vector $x \in F_1 \cap k^n$ such that $\beta|_{F_1}(x,x) = 0$. So, as in the last case, let $a_2 = x$ and $b_2 = Ax$. As before, we have $\beta(a_1+ib_1, a_1-ib_1) \ne 0$.

Let $E_2 = \Span_{k[i]}(a_1,a_2,b_1,b_2)$, and let $F_2$ be the orthogonal complement of $E_2$ over $k[i]$. In this manner, we can create the bases that we noted in the opening paragraph of this proof. 

Note that we always have $$0 = \beta(a_j+ib_j, a_l+ib_l) = (\beta(a_j,a_l)-\beta(b_j,b_l))+i(\beta(a_j,b_l)+\beta(b_j,a_l)),$$ and when $j \ne l$ we have $$0 = \beta(a_j+ib_j, a_l-ib_l) = (\beta(a_j,a_l)+\beta(b_j,b_l))+i(-\beta(a_j,b_l)+\beta(b_j,a_l)).$$

This tells us that when $j \ne l$ that $$\beta(a_j,b_l) = \beta(a_j,a_l) = \beta(b_j,b_l) = 0.$$ When $j = l$, we see that $\beta(b_j,b_j) = \beta(a_j,a_j)$ and that $\beta(a_j, b_j)= -\beta(b_j,a_j)$. The last of these shows that $\beta(a_j,b_l) = 0$, regardless of the values of $j$ and $l$.

Let $$U = (a_1,...,a_{\frac{n}{2}},b_1,...,b_{\frac{n}{2}}).$$ Then, it follows that $U^TMU = \left(\begin{smallmatrix}U_1 & 0 \\0 & U_1\end{smallmatrix}\right)$ where $U_1$ is a diagonal $\frac{n}{2} \times \frac{n}{2}$ matrix.

Lastly, since $b_j = Aa_j$, then it follows that $Ab_j = -a_j$. So, we have that  $$A = U\left(\begin{array}{cc}0 & -I_{\frac{n}{2}} \\ I_{\frac{n}{2}} & 0\end{array}\right) U^{-1}.$$
\end{proof}

We now look at an example that highlights some of these results that we have just proven in the Type 3 case.

\begin{beisp}
Assume that $\beta$ is the standard dot product. Then, $\theta$ can be a Type 3 $k$-involution of $\So(4,\mathbb{R})$ only if we can choose $A \in \oo(4,\mathbb{R})$ such that $A^2 =-I$. This means the matrix must satisfy $A^T = -A$. That is, the matrix must be skew-symmetric. Observe that the matrix $$A = \left(\begin{array}{cccc}0 & 1 & 0 & 0 \\-1 & 0 & 0 & 0 \\0 & 0 & 0 & 1 \\0 & 0 & -1 & 0\end{array}\right) \in \oo(4, \mathbb{R})$$ is skew-symmetric, so it induces a Type 3 $k$-involution of $\So(4, \mathbb{R})$. It can be shown that $E(A,-i)$ has dimension 2. A basis for this subspace is formed by the vectors $$v_1 =  \left(\begin{array}{c}0\\ 0 \\0 \\1\end{array}\right)+ i \left(\begin{array}{c}0 \\ 0 \\1 \\0\end{array}\right)$$ and $$ v_2=\left(\begin{array}{c} 0 \\ 1 \\0 \\0\end{array}\right)+ i\left(\begin{array}{c}1 \\ 0 \\ 0 \\0\end{array}\right).$$ 

It can be shown that $$v_3 =  \left(\begin{array}{c}0\\ 0 \\0 \\1\end{array}\right)- i \left(\begin{array}{c}0 \\ 0 \\1 \\0\end{array}\right)$$ and $$ v_4=\left(\begin{array}{c} 0 \\ 1 \\0 \\0\end{array}\right)- i\left(\begin{array}{c}1 \\ 0 \\0 \\0\end{array}\right)$$  are $i$-eigenvectors of $A$, where these are the conjugates of $v_1$ and $v_2$, respectively.  

Following the notation of the previous lemma, we have $$U = \left(\begin{array}{cccc}0 & 0 & 0 & 1 \\0 & 1 & 0 & 0 \\0 & 0 & 1 & 0 \\1 & 0 & 0 & 0\end{array}\right),$$ where $U^TU = I_4$ and $A = U \left(\begin{array}{cc}0 & -I_{\frac{n}{2}} \\ I_{\frac{n}{2}} & 0\end{array}\right) U^{-1}$ . Using the notation of Lemma \ref{Type3ClassNoSo}, we note that $U_1 = I_2$.
\end{beisp}

We now find conditions on Type 3 $k$-involutions that are equivalent to isomorphy, in the case that $i \not \in k$.

\begin{theorem}
\label{type3lemNoSo}
Assume $i \not \in k$. Then, if $\Inn_A$ and $\Inn_B$ are both Type 3 $k$-involutions of $\So(n,k,\beta)$, then $\Inn_A$ and $\Inn_B$ are isomorphic over $\oo(n,k,\beta)$.
\end{theorem}

\begin{proof}

By the Lemma \ref{Type3ClassNoSo}, we can choose a matrix $U \in \Gl(n,k)$ such that $$A = U \left(\begin{array}{cc}0 & -I_{\frac{n}{2}} \\ I_{\frac{n}{2}} & 0\end{array}\right) U^{-1}
\text{ for }
U  = \left(\begin{array}{cccccccccc}a_1 & a_2 & \cdots & a_\frac{n}{2} &b_1 & b_2 & \cdots & b_\frac{n}{2}  \end{array}\right)\in \Gl(n,k),$$ where the $a_j+ib_j$ are a basis for $E(A,-i)$, the $a_j-ib_j$ are a basis for $E(A,i)$, and $U^TMU = \left(\begin{smallmatrix}U_1 & 0 \\ 0 & U_1\end{smallmatrix}\right)$ is a diagonal matrix. 

Let $$X = (a_1+ib_1,...,a_{\frac{n}{2}}+ib_{\frac{n}{2}}, a_1-ib_1,...,a_{\frac{n}{2}}-ib_{\frac{n}{2}}),$$ and consider $\Inn_A$ and $\Inn_B$ as $k$-involutions of $\So(n,k[i],\beta)$. By construction, we see that $X$ is a matrix that satisfies the conditions of Lemma \ref{Type3ClassYesSo} for the group $\So(n,k[i],\beta)$. We note that $X_1 = 2U_1$. We also know by Theorem \ref{type3lemYesSo} that $\Inn_A$ and $\Inn_B$ are isomorphic (when viewed as $k$-involutions of $\So(n,k[i],\beta)$) over $\oo(n,k[i],\beta)$. So, we can choose $Q_i \in \oo(n,k[i],\beta)$ such that $Q_i^{-1}AQ_i = B$. Let $Y = Q_i^{-1}X$. We now show a couple of facts about $Y$.

First, we note that since $Y$ was obtained from $X$ via row operations, then for $1 \le j \le \frac{n}{2}$, the $j$th and $\frac{n}{2}+j$th columns are $i$-conjugates of one another.

Also, note that 
\begin{align*}
Y^{-1}BY &= (Q_i^{-1}X)^{-1}B (Q_i^{-1}X)\\ 
&= X^{-1}Q_iBQ_i^{-1}X\\
&= X^{-1}AX\\ 
&= \left(\begin{array}{cc} -iI_{\frac{n}{2}} & 0 \\0&  iI_{\frac{n}{2}} \end{array}\right).
\end{align*}

Lastly, we see that 
\begin{align*}
Y^TMY &= (Q_i^{-1}X)^TM(Q_i^{-1}X)\\ 
& = X^T((Q_i^{-1})^TMQ_i)X\\
&= X^TMX\\
&= \left(\begin{array}{cc} 0 & X_1 \\ X_1 & 0 \end{array}\right)\\ 
&= \left(\begin{array}{cc} 0 & 2U_1 \\ 2U_1 & 0 \end{array}\right).
\end{align*}

We can write $$Y = (c_1+id_1,...,c_{\frac{n}{2}}+id_{\frac{n}{2}}, c_1-id_1,...,c_{\frac{n}{2}}-id_{\frac{n}{2}})$$ where $c_j, d_j \in k^n$. So, let $$V = (c_1,...,c_{\frac{n}{2}},d_1,...,d_{\frac{n}{2}})\in \Gl(n,k).$$ It follows from what we have shown that $B = V \left(\begin{smallmatrix}0 & -I_{\frac{n}{2}} \\ I_{\frac{n}{2}} & 0\end{smallmatrix}\right) V^{-1}$ where $$V^TMV = \left(\begin{array}{cc}U_1 & 0 \\ 0 & U_1\end{array}\right) = U^TMU.$$

Now, let $Q = UV^{-1}$. We will show that $Q^{-1}AQ = B$ and $Q \in \oo(n,k,\beta)$. This will prove that $\Inn_A$ and $\Inn_B$ are isomorphic over $\oo(n,k,\beta)$ by Lemma \ref{TidyLem}. 

We first show that $Q \in \oo(n,k,\beta)$. 
\begin{align*}
Q^TMQ &= (UV^{-1})^TMUV^{-1}\\ 
&= (V^{-1})^T(U^TMU)V^{-1}\\ 
&= (V^{-1})^T(V^TMV)V^{-1}\\ 
&= M.
\end{align*}

Lastly, we show that $Q^{-1}AQ = B$. 
\begin{align*}
Q^{-1}AQ &= (UV^{-1})^{-1}A(UV^{-1})\\ 
&= VU^{-1}AUV^{-1}\\
&= V\left(\begin{array}{cc}0 & -I_{\frac{n}{2}} \\ I_{\frac{n}{2}} & 0\end{array}\right) V^{-1}\\ 
&= B.
\end{align*}

\end{proof}

Combining the results from this section, we get the following corollary. 

\begin{cor}
\label{CorType3So}
If $\Inn_A$ and $\Inn_B$ are both Type 3 $k$-involutions of $\So(n,k,\beta)$ where $A, B \in \oo(n,k \beta)$, then $\Inn_A$ and $\Inn_B$ are isomorphic over $\oo(n,k,\beta)$. That is, $\So(n,k,\beta)$ has at most one isomorphy class of Type 3 $k$-involutions.
\end{cor}

If $\Inn_A$ and $\Inn_B$ are both Type 3 $k$-involutions of $\So(n,k,\beta)$ where $A, B \in \oo(n,k \beta)$, then $\Inn_A$ and $\Inn_B$ are isomorphic over $\oo(n,k,\beta)$. That is, $\So(n,k,\beta)$ has at most one isomorphy class of Type 3 $k$-involutions.

\subsection{Type 4 $k$-involutions}

We now move on to a similar characterization in the Type 4 case. First, we characterize the eigenvectors of the matrices that induce these $k$-involutions. Recall that we can choose $A \in \oo(n,k[\sqrt{\alpha}],\beta)$ such that each entry of $A$ is a $k$-multiple of $\sqrt{\alpha}$, and that we know $A^2 = -I$. We begin by proving a couple of lemmas about the eigenspaces of these matrices.

\begin{lem}

Suppose $A \in \oo(n,k[\sqrt{\alpha}],\beta) $ induces a Type 4 $k$-involution of $\So(n,k,\beta)$. Also suppose $x,y \in k^n$ such that $x+\sqrt{-\alpha}y \in E(A,-i)$. Then, $x-\sqrt{-\alpha}y \in E(A,i)$. Likewise, if $u,v \in k^n$ such that $u+\sqrt{-\alpha}v \in E(A,i)$. Then, $u-\sqrt{-\alpha}v \in E(A,-i)$. Further, $\dim(E(A,i))= \dim(E(A,-i))$.

\end{lem}

\begin{proof}

Suppose $x,y \in k^n$ such that $x+\sqrt{-\alpha}y \in E(A,-i)$. Then,
$$A(x+\sqrt{-\alpha}y) = -i(x+\sqrt{-\alpha}y)$$ which implies
$$Ax+\sqrt{-\alpha}Ay = \sqrt{\alpha}y-ix.$$
Then, complex conjugation tells us that 
$$Ax-\sqrt{-\alpha}Ay = \sqrt{\alpha}y+ix,$$ which tells us that
$$A(x-\sqrt{-\alpha}y) = i(x-\sqrt{-\alpha}y).$$ A similar argument shows that if $u,v \in k^n$ such that $u+\sqrt{-\alpha}v \in E(A,i)$. Then, $u-\sqrt{-\alpha}v \in E(A,-i)$.

Since  $x+\sqrt{-\alpha}y \in E(A,-i)$ implies $x-\sqrt{-\alpha}y \in E(A,i)$ and vice versa, then we see that $\dim(E(A,i))= \dim(E(A,-i))$. 
\end{proof}

\begin{lem}
\label{Type4EigenSo}
Suppose $\theta = \Inn_A$ is a Type 4 $k$-involution of $\So(n,k,\beta)$ where $A \in $ \newline$\oo(n,k[\sqrt{\alpha}],\beta)$. Then, we can find $x_1,...,x_{\frac{n}{2}}, y_1,...,y_{\frac{n}{2}} \in k^n$ such that the $x+\sqrt{-\alpha}y$ are a basis for $E(A,-i)$ and the $x-\sqrt{-\alpha}y$ are a basis for $E(A,i)$.
\end{lem}

\begin{proof}

Since $\Inn_A$ is Type 4, then we are assuming that $A \in \oo(n,k[\sqrt{\alpha}], \beta)$ and $A^2 = -I$. Note that this also means that $n$ is even. It follows that all eigenvalues of $A$ are $\pm i$. Since there are no repeated roots in the minimal polynomial of $A$, then we see that $A$ is diagonalizable. We begin by constructing bases for $E(A,i)$ and $E(A,-i)$ such that all the basis vectors lie in $k[i]^n$.  Let $\{z_1,...,z_n\}$ be a basis for $k^n$. For each $j$, let $u_j = (\sqrt{\alpha}A-\sqrt{-\alpha}I)z_j.$ Note that 
\begin{align*}
Au_j &= A(\sqrt{\alpha}A-\sqrt{-\alpha}I)z_j\\
&= (\sqrt{\alpha}A^2-\sqrt{-\alpha}A)z_j\\ 
&= -i(\sqrt{\alpha}A-\sqrt{-\alpha}I)z_j\\
&= -iu_j.
\end{align*}
 So, $\{u_1,...,u_n\}$ must span $E(A,-i)$. Thus, we can appropriately choose $\frac{n}{2}$ of these vectors and form a basis for $E(A,-i)$. Note that each of these vectors lies in $k[i]^n$. Label these basis vectors as $v_1,...,v_\frac{n}{2}$. We can write each of these vectors as $v_j = x_j+\sqrt{-\alpha}y_j$. By the previous lemma, we know that $x_j-\sqrt{-\alpha}y_j \in E(A,i)$, and it follows that these will be linearly independent. Since there are $\frac{n}{2}$ of them, then they form a basis for $E(A,i)$.
\end{proof}

We are now able to prove results that characterize the matrices that induce Type 4 $k$-involutions, and then use these characterizations to find conditions on these $k$-involutions that are equivalent to isomorphy. We will have separate cases, depending on whether or not $\sqrt{-\alpha}$ lies in $k$. We begin by assuming that $\sqrt{-\alpha} \in k$. Since we are also assuming that $\sqrt{\alpha} \not \in k$, then it follows from these two assumptions that $\alpha$ and $-1$ lie in the same square class of $k$. Thus, we can assume in this case that $\alpha = -1$, which means $\sqrt{-\alpha} = 1$.

\begin{lem}
\label{Type4ClassYesSo}
Assume $\sqrt{-\alpha} \in k$ and suppose $\theta=\Inn_A$ is a Type 4 $k$-involution of $\So(n,k,\beta)$. Then, $A = X \left(\begin{smallmatrix}-iI_{\frac{n}{2}} &0 \\ 0& iI_{\frac{n}{2}}\end{smallmatrix}\right) X^{-1}$ for some $X \in \Gl(n,k),$ where $X^TMX = \left(\begin{smallmatrix}0 & X_1 \\ X_1 & 0\end{smallmatrix}\right)$, and $X_1$ is diagonal.
\end{lem}

\begin{proof}

We know from Lemma \ref{Type4EigenSo} that we have bases for $E(A,-i)$ and $E(A,I)$ that lie in $k^n$. We will show that we can in fact choose bases $a_1,...,a_{\frac{n}{2}}$ for $E(A,-i) \cap k^n$ and $b_1,...,b_{\frac{n}{2}}$ for $E(A,i) \cap k^n$ such that $\beta(a_j,a_l) = 0 = \beta(b_j,b_l)$ and $\beta(a_j, b_l)$ is nonzero if and only if $j = l$. We will build these bases recursively.

First, we know that we can choose some nonzero $a_1 \in E(A,-i) \cap k^n$. Then, since $\beta$ is non degenerate, we can choose a vector $t$ such that $\beta(a_1, t) \ne 0$. We note that $E(A,-i) \oplus E(A,i) = k^n$, so we can choose $t_{-i} \in E(A,-i)\cap k^n$ and $t_i \in E(A,i)\cap k^n$ such that $t = t_{-i}+t_i$. Since $\beta(a_1, t_{-i}) = 0$, then it follows that $\beta(a_1, t_i) \in k$ is nonzero. Let $b_1 = t_i$.

Let $E_1 = \Span_k(a_1,b_1)$ and let $F_1$ be the orthogonal complement of $E_1$ in $k^n$. Since the system of linear equations $$\beta(a_1,x) = 0$$ $$\beta(b_1,x) =0$$ has $n-2$ free variables, then we see that $F_1$ has dimension $n-2$. 

We now want to find $a_2 \in F_1 \cap E(A,-i)$. Similar to the construction in the previous lemma, we can choose $x \in F_1$, and let $a_2 = (\sqrt{\alpha}A-\sqrt{-\alpha}I)x$. It follows that $a_2 \in F_1 \cap E(A,-i)$. Now we want $b_2 \in F_2 \cap E(A,i)$ such that $\beta(a_2, b_2)$ is nonzero. Since $\beta|_{F_1}$ is non degenerate, then there exists some $y \in F_2$ such that $\beta(a_2,y) \ne 0$. Similar to the construction of $b_1$, we see that this implies the existence a vector $b_2$ that fits our criteria. 

Now, we let $E_2 = \Span_k(a_1,a_2,b_1,b_2)$ and let $F_2$ be the orthogonal complement of $E_2$ in $k^n$. We continue this same argument $\frac{n}{2}$ times, until we have the bases that we wanted to find. Let $$X= (a_1,...,a_{\frac{n}{2}}, b_1,...,b_{\frac{n}{2}}).$$ Then, the result follows.
\end{proof}

Here is an example of a Type 4 $k$-involution when $\sqrt{-\alpha} \in k$.

\begin{beisp}
Assume that $\beta$ is the standard dot product and that $k = \mathbb{F}_3$, the field of three elements. So, the square roots of 2 are $\pm i$. Observe that the matrix $$A =i\left(\begin{array}{cccc}0 & 0 & 1 & 1 \\0 & 0 & 1 & -1 \\2 & 2 & 0 & 0 \\2 & 1 & 0 & 0\end{array}\right) \in \oo(4, \mathbb{F}_3[i])$$ satisfies the relation $A^2 = -I_4$. Since each entry of $A$ is a $\mathbb{F}_3$-multiple of $i$, then it follows that $\Inn_A$ is an $k$-involution of $\So(4,\mathbb{F}_3)$ of Type 4. A basis for $E(A,-i)$ is formed by the vectors $$v_1 =  \left(\begin{array}{c}0\\ 0 \\0 \\1\end{array}\right)+ \left(\begin{array}{c}  1 \\ 2\\ 0 \\ 0\end{array}\right) =  \left(\begin{array}{c}  1 \\ 2\\ 0 \\ 1\end{array}\right)$$ and $$ v_2=\left(\begin{array}{c} 0 \\ 0 \\1 \\0\end{array}\right)+ \left(\begin{array}{c} 1\\ 1\\0 \\ 0 \end{array}\right) =  \left(\begin{array}{c}  1 \\ 1\\ 1 \\ 0\end{array}\right).$$ 

It can be shown that  $$v_3 =  \left(\begin{array}{c}0\\ 0 \\0 \\1\end{array}\right)- \left(\begin{array}{c}  1 \\ 2\\ 0 \\ 0\end{array}\right) =  \left(\begin{array}{c}  2 \\ 1\\ 0 \\ 1\end{array}\right)$$ and $$ v_2=\left(\begin{array}{c} 0 \\ 0 \\1 \\0\end{array}\right)- \left(\begin{array}{c} 1\\ 1\\0 \\ 0 \end{array}\right) =  \left(\begin{array}{c}  2 \\ 2\\ 1 \\ 0\end{array}\right)$$    are $i$-eigenvectors of $A$.

Following the notation of the previous lemma, we have $$X = \left(\begin{array}{cccc}0 & 0 & 1 & 1 \\0 & 0 & 2 &1 \\0 & 1 & 0 & 0 \\1 & 0 & 0 & 0\end{array}\right),$$ where $X^TX = \left(\begin{smallmatrix}1 & 0 & 0 & 0 \\0 & 1 & 0 & 0 \\0 & 0 & 2 & 0 \\0 & 0 & 0 & 2\end{smallmatrix}\right)$ and $A = -iX \left(\begin{smallmatrix}0 & 0 & 1 & 0 \\0 & 0 & 0 & 1 \\-i & 0 & 0 & 0 \\0 & -i & 0 & 0\end{smallmatrix}\right) X^{-1}$. We also note that $X_1 = I$.

\end{beisp}

Now we characterize the isomorphy classes of Type 4 $k$-involutions in the case where $\sqrt{-\alpha} \in k$.

\begin{theorem}
\label{type4lemYesSo}
Assume that $\sqrt{-\alpha} \in k$. Then, if $\Inn_A$ and $\Inn_B$ are both Type 4 $k$-involutions of $\So(n,k,\beta)$ where the entries of $A$ and $B$ are $k$-multiples of $\sqrt{\alpha}$, then $\Inn_A$ and $\Inn_B$ are isomorphic over $\oo(n,k,\beta)$.
\end{theorem}

\begin{proof}

Suppose we have two such $k$-involutions of $\So(n,k,\beta)$. Let them be represented by matrices $A,B \in \oo(n,k,\beta)$. By Lemma \ref{Type4ClassYesSo}, we can choose $X, Y \in \Gl(n,k)$ such that $$X^{-1}AX = \left(\begin{array}{cc}-iI & 0 \\0 & iI\end{array}\right) = Y^{-1}BY,$$ $$X^TMX = \left(\begin{array}{cc}0 & X_1\\ X_1 & 0\end{array}\right),$$ and  $$Y^TMY = \left(\begin{array}{cc}0 & Y_1\\ Y_1 & 0\end{array}\right),$$ where $X_1$ and $Y_1$ are diagonal.

Since $X_1$ and $Y_1$ are both invertible diagonal matrices, then we can choose $R_1$ and $R_2 \in \Gl(\frac{n}{2},k)$ such that $Y_1 = R_1^TX_1R_2$. Let $R = \left(\begin{smallmatrix}R_1 & 0 \\0 & R_2\end{smallmatrix}\right)$ and $Q = XRY^{-1}$. It follows from this that $R^TX^TMXR = Y^TMY$. We will show that $Q \in \oo(n,k,\beta)$ and $Q^{-1}AQ = B$. This will then prove that $\Inn_A$ and $\Inn_B$ lie in the same isomorphy class by Lemma \ref{TidyLem}. 

First we show that $Q \in \oo(n,k,\beta)$. By construction, the entries of $Q$ lie in $k$. Also, note that 
\begin{align*}
Q^TMQ &= (XRY^{-1})^TM(XRY^{-1})\\ 
&= (Y^{-1})^TR^T(X^TMX)RY^{-1}\\
&= (Y^{-1})^T(Y^TMY)Y^{-1}\\ 
&= M,
\end{align*} 
which proves  $Q \in \oo(n,k,\beta)$. 

Lastly, we show that $Q^{-1}AQ = B$. We first note that $R$ and $\left(\begin{smallmatrix}-iI & 0 \\0 & iI\end{smallmatrix}\right)$ commute. Then, we see that
\begin{align*}
Q^{-1}AQ &= (XRY^{-1})^{-1}A (XRY^{-1})\\
&= YR^{-1}(X^{-1}AX)RY^{-1}\\
&= Y R^{-1}\left(\begin{array}{cc}-iI & 0 \\0 & iI\end{array}\right)R Y^{-1}\\ 
&= Y R^{-1}R\left(\begin{array}{cc}-iI & 0 \\0 & iI\end{array}\right) Y^{-1}\\
&= Y \left(\begin{array}{cc}-iI & 0 \\0 & iI\end{array}\right) Y^{-1}\\
&= B.
\end{align*}
\end{proof}

We now examine the case where $\sqrt{-\alpha} \not \in k$.

\begin{lem}
\label{Type4ClassNoSo}
Assume $\sqrt{-\alpha} \not \in k$ and suppose $\theta=\Inn_A$ is a Type 4 $k$-involution of $\So(n,k,\beta)$. Then, $A = -\frac{\sqrt{\alpha}}{\alpha} U \left(\begin{smallmatrix}0 &  I_{\frac{n}{2}} \\ -\alpha I_{\frac{n}{2}} & 0\end{smallmatrix}\right) U^{-1}$ for $$U  = \left(\begin{array}{cccccccccc}a_1 & a_2 & \cdots & a_\frac{n}{2} &b_1 & b_2 & \cdots & b_\frac{n}{2}  \end{array}\right)\in \Gl(n,k),$$ where the $a_j+\sqrt{-\alpha}b_j$ are a basis for $E(A,-i)$, the $a_j-\sqrt{-\alpha}b_j$ are a basis for $E(A,i)$, and $U^TMU = \left(\begin{smallmatrix}U_1 & 0 \\0 & \frac{1}{\alpha}U_1\end{smallmatrix}\right)$ is diagonal.
\end{lem}

\begin{proof}

We know from Lemma \ref{Type4EigenSo} that we have bases for $E(A,-i)$ and $E(A,I)$ that lie in $k[\sqrt{-\alpha}]^n$. We will show that we can in fact choose bases $a_1+\sqrt{-\alpha}b_1,...,a_{\frac{n}{2}}+\sqrt{-\alpha}b_{\frac{n}{2}}$ for $E(A,-i) \cap k[i]^n$ and $a_1-\sqrt{-\alpha}b_1,...,a_{\frac{n}{2}}-\sqrt{-\alpha}b_{\frac{n}{2}}$ for $E(A,i) \cap k[\sqrt{-\alpha}]^n$ such that $\beta(a_j+\sqrt{-\alpha}b_j, a_l-\sqrt{-\alpha}b_l)$ is nonzero if and only if $j=l$. From this, we will be able to show that $\beta(a_j,a_l) = 0 = \beta(b_j,b_l)$ when $j \ne l$ and $\beta(a_j, b_l) = 0$ for all $j$ and $l$. We will build these bases recursively.

Given any vector $x \in k^n$, we know that $x+iAx \in E(A,-i)$. We want to choose $x\in k^n$ such that $\beta(x,x) \ne 0$. (The reasons for this will become apparent.) $M$ is an invertible matrix, so there are at least $n$ instances of $e_j^TMe_l \ne 0$. If there is an instance where $j=l$, let $x=e_j$. If instead we have  $e_j^TMe_l = 0 = e_l^TMe_j$, then let $x = e_j+e_l$. We note that this works because $$\beta(x,x) = \beta(e_j+e_l, e_j+e_l) = 2 \beta(e_j,e_l) \ne 0.$$ 

So, we have $x \in k^n$ such that $\beta(x,x) \ne 0$, and we have $x+iAx \in E(A,-i)$. Let $a_1 = x$ and $b_1 = \frac{1}{\sqrt{\alpha}}Ax$. So, $a_1+\sqrt{-\alpha}b_1 \in E(A,-i)$ and $a_1-\sqrt{-\alpha}b_1 \in E(A,i)$.  From this, it follows that 
\begin{align*}
\beta(a_1+\sqrt{-\alpha}b_1, a_1-\sqrt{-\alpha}b_1) &= (\beta(a_1,a_1) +\alpha \beta(b_1,b_1))+\sqrt{-\alpha}(-\beta(a_1,b_1)+\beta(a_1,b_1)\\
&= \beta(x,x) +\alpha \beta\left( \frac{1}{\sqrt{\alpha}}Ax,\frac{1}{\sqrt{\alpha}}Ax\right)\\ 
&= 2\beta(x,x)\\ 
&\ne 0.
\end{align*}

Let $E_1 = \Span_{k[\sqrt{-\alpha}]}(a_1+\sqrt{-\alpha}b_1, a_1-\sqrt{-\alpha}b_1) = \Span_{k[\sqrt{-\alpha}]}(a_1,b_1)$, and let $F_1$ be the orthogonal complement of $E_1$ over $k[\sqrt{-\alpha}]$. $F_1$ has dimension $n-2$, and $\beta|_{F_1}$ is nondegenerate. So, we can find a nonzero vector $x \in F_1 \cap k^n$ such that $\beta|_{F_1}(x,x) = 0$. So, as in the last case, let $a_2 = x$ and $b_2 = \frac{1}{\sqrt{\alpha}}Ax$. As before, we have $\beta(a_2+\sqrt{-\alpha}b_2, a_2-\sqrt{-\alpha}b_2) \ne 0$.

Let $E_2 = \Span_{k[\sqrt{-\alpha}]}(a_1,a_2,b_1,b_2)$, and let $F_2$ be the orthogonal complement of $E_2$ over $k[\sqrt{-\alpha}]$. In this manner, we can create the bases that we noted in the opening paragraph of this proof. 

Note that we always have $$0 = \beta(a_j+\sqrt{-\alpha}b_j, a_l+\sqrt{-\alpha}b_l) = (\beta(a_j,a_l)-\alpha \beta(b_j,b_l))+\sqrt{-\alpha}(\beta(a_j,b_l)+\beta(b_j,a_l)),$$ and when $j \ne l$ we have $$0 = \beta(a_j+\sqrt{-\alpha}b_j, a_l-\sqrt{-\alpha}b_l) = (\beta(a_j,a_l)+\alpha\beta(b_j,b_l))+\sqrt{-\alpha}(-\beta(a_j,b_l)+\beta(b_j,a_l)).$$

This tells us that when $j \ne l$ that $$\beta(a_j,b_l) = \beta(a_j,a_l) = \beta(b_j,b_l) = 0.$$ When $j = l$, we see that $\beta(b_j,b_j) = \frac{1}{\alpha}\beta(a_j,a_j)$ and that $\beta(a_j, b_j)= -\beta(b_j,a_j)$. The last of these shows that $\beta(a_j,b_l) = 0$, regardless of the values of $j$ and $l$.

Let $$U = (a_1,...,a_{\frac{n}{2}},b_1,...,b_{\frac{n}{2}}).$$ Then, it follows that $U^TMU = \left(\begin{smallmatrix}U_1 & 0 \\0 & \frac{1}{\alpha}U_1\end{smallmatrix}\right)$ where $U_1$ is a diagonal $\frac{n}{2} \times \frac{n}{2}$ matrix.

Lastly, since $b_j = \frac{1}{\sqrt{\alpha}}Aa_j$, then it follows that $Ab_j = -\frac{1}{\sqrt{\alpha}}a_j$. So, we have that  $A = -\frac{\sqrt{\alpha}}{\alpha}U \left(\begin{smallmatrix}0 & I_{\frac{n}{2}} \\ -\alpha I_{\frac{n}{2}} & 0\end{smallmatrix}\right) U^{-1}$.
\end{proof}

Here is an example of a Type 4 $k$-involution in the case that $\sqrt{-\alpha} \not \in k$.

\begin{beisp}
Assume that $\beta$ is the standard dot product. Observe that the matrix $$A = \frac{\sqrt{2}}{2}\left(\begin{array}{cccc}0 & 0 & 1 & 1 \\0 & 0 & 1 & -1 \\-1 & -1 & 0 & 0 \\-1 & 1 & 0 & 0\end{array}\right) \in \oo(4, \mathbb{Q}[\sqrt{2}])$$ is such that $A^2 = -I_4$. Since each entry of $A$ is a $\mathbb{Q}$-multiple of $\sqrt{2}$, then it follows that $\Inn_A$ is an $k$-involution of $\So(4,\mathbb{Q})$ of Type 4. It can be shown that $E(A,-i)$ has dimension 2. A basis for this subspace is formed by the vectors $$v_1 =  \left(\begin{array}{c}0\\ 0 \\0 \\1\end{array}\right)+ \sqrt{-2} \left(\begin{array}{c}  -\frac{1}{2} \\ \frac{1}{2}\\ 0 \\ 0\end{array}\right)$$ and $$ v_2=\left(\begin{array}{c} 0 \\ 0 \\1 \\0\end{array}\right)+ \sqrt{-2} \left(\begin{array}{c} -\frac{1}{2} \\ -\frac{1}{2}\\0 \\ 0 \end{array}\right).$$ 

It can be shown that$$v_3 =  \left(\begin{array}{c}0\\ 0 \\0 \\1\end{array}\right)- \sqrt{-2} \left(\begin{array}{c}  -\frac{1}{2} \\ \frac{1}{2}\\ 0 \\ 0\end{array}\right)$$ and $$ v_4=\left(\begin{array}{c} 0 \\ 0 \\1 \\0\end{array}\right)- \sqrt{-2} \left(\begin{array}{c} -\frac{1}{2} \\ -\frac{1}{2}\\0 \\ 0 \end{array}\right)$$   are $i$-eigenvectors of $A$, where these are the conjugates of $v_1$ and $v_2$, respectively.  

Following the notation of the previous lemma, we have $$U = \left(\begin{array}{cccc}0 & 0 & -\frac{1}{2} & -\frac{1}{2} \\0 & 0 & \frac{1}{2} &- \frac{1}{2} \\0 & 1 & 0 & 0 \\1 & 0 & 0 & 0\end{array}\right),$$ where $U^TU = \left(\begin{smallmatrix}1 & 0 & 0 & 0 \\0 & 1 & 0 & 0 \\0 & 0 & \frac{1}{2} & 0 \\0 & 0 & 0 & \frac{1}{2}\end{smallmatrix}\right)$ and $A = -\frac{\sqrt{2}}{2}U \left(\begin{smallmatrix}0 & 0 & 1 & 0 \\0 & 0 & 0 & 1 \\-\sqrt{2} & 0 & 0 & 0 \\0 & -\sqrt{2} & 0 & 0\end{smallmatrix}\right) U^{-1}$ . We also note that $U_1 = I$.
\end{beisp}

We now find conditions on Type 4 $k$-involutions that are equivalent to isomorphy in the case where $\sqrt{-\alpha} \not \in k$.

\begin{theorem}
\label{type4lemNoSo}
Assume $\sqrt{-\alpha} \not \in k$. Then, if $\Inn_A$ and $\Inn_B$ are both Type 4 $k$-involutions of $\So(n,k,\beta)$ where $A, B \in \oo(n,k[\sqrt{\alpha}], \beta)$, then $\Inn_A$ and $\Inn_B$ are isomorphic over $\oo(n,k,\beta)$.
\end{theorem}

\begin{proof}

By Lemma \ref{Type4ClassNoSo}, we can choose a matrix $U \in \Gl(n,k)$ such that $$A = -\frac{\sqrt{\alpha}}{\alpha} U \left(\begin{array}{cc}0 &  I_{\frac{n}{2}} \\ -\alpha I_{\frac{n}{2}} & 0\end{array}\right) U^{-1}$$ for $$U  = \left(\begin{array}{cccccccccc}a_1 & a_2 & \cdots & a_\frac{n}{2} &b_1 & b_2 & \cdots & b_\frac{n}{2}  \end{array}\right),$$ where the $a_j+\sqrt{-\alpha}b_j$ are a basis for $E(A,-i)$, the $a_j-\sqrt{-\alpha}b_j$ are a basis for $E(A,i)$, and $U^TMU = \left(\begin{smallmatrix}U_1 & 0 \\0 & \frac{1}{\alpha}U_1\end{smallmatrix}\right)$ is diagonal.

  Consider $\Inn_A$ and $\Inn_B$ as $k$-involutions of $\So(n,k[\sqrt{-\alpha}],\beta)$. If $k[\sqrt{-\alpha}] = k[\sqrt{\alpha}]$, then these are Type 3 $k$-involutions of $\So(n,k[\sqrt{-\alpha}],\beta)$, since $A$ and $B$ would have entries in the field, and $i \in k[\sqrt{-\alpha}]$. Otherwise, if $k[\sqrt{-\alpha}] \ne k[\sqrt{\alpha}]$, then these are Type 4 $k$-involutions where $ \sqrt{-\alpha} \in k[\sqrt{-\alpha}]$.

Let 
$$X  = (a_1+\sqrt{-\alpha}b_1,...,a_{\frac{n}{2}}+\sqrt{-\alpha}b_{\frac{n}{2}}, a_1-\sqrt{-\alpha}b_1,...,a_{\frac{n}{2}}-\sqrt{-\alpha}b_{\frac{n}{2}}).$$ 

By construction, we see that $X$ is a matrix that satisfies the conditions of Lemma \ref{Type3ClassNoSo} or Lemma \ref{Type4ClassYesSo} for the group $\So(n,k[\sqrt{\alpha}],\beta)$. We note that $X_1 = 2U_1$. We also know by Corollary \ref{CorType3So} or Theorem \ref{type4lemYesSo} that $\Inn_A$ and $\Inn_B$ are isomorphic (when viewed as $k$-involutions of $\So(n,k[\sqrt{-\alpha}],\beta)$) over $\oo(n,k[\sqrt{-\alpha}],\beta)$. So, we can choose $Q_{\alpha} \in \oo(n,k[\sqrt{-\alpha}],\beta)$ such that $Q_{\alpha}^{-1}AQ_{\alpha} = B$. Let $Y = Q_{\alpha}^{-1}X$. Since $Y$ is constructed by doing row operations on $X$, then we can write 
$$Y  = (c_1+\sqrt{-\alpha}d_1,...,c_{\frac{n}{2}}+\sqrt{-\alpha}d_{\frac{n}{2}}, c_1-\sqrt{-\alpha}d_1,...,c_{\frac{n}{2}}-\sqrt{-\alpha}c_{\frac{n}{2}}),$$ 
where $c_j, d_j \in k^n$. We now show a couple of facts about $Y$.

First, we note that since $Y$ was obtained from $X$ via row operations, then for $1 \le j \le \frac{n}{2}$, the $j$th and $\frac{n}{2}+j$th columns are $i$-conjugates of one another.

Next, we observe that 
\begin{align*}
Y^{-1}BY &= (Q_{\alpha}^{-1}X)^{-1}B (Q_{\alpha}^{-1}X)\\ 
&= X^{-1}Q_{\alpha}BQ_{\alpha}^{-1}X \\
&= X^{-1}AX\\ 
&= \left(\begin{array}{cc} -iI_{\frac{n}{2}} & 0 \\0&  iI_{\frac{n}{2}} \end{array}\right).
\end{align*}

Lastly, we see that 
\begin{align*}
Y^TMY &= (Q_{\alpha}^{-1}X)^TM(Q_{\alpha}^{-1}X)\\
&= X^T((Q_{\alpha}^{-1})^TMQ_{\alpha})X\\
&= X^TMX\\
&= \left(\begin{array}{cc} 0 & X_1 \\ X_1 & 0 \end{array}\right)\\ 
&= \left(\begin{array}{cc} 0 & 2U_1 \\ 2U_1 & 0 \end{array}\right).
\end{align*}

Let $$V = (c_1,...,c_{\frac{n}{2}},d_1,...,d_{\frac{n}{2}}) \in \Gl(n,k).$$ It follows from what we have shown that $B = -\frac{\sqrt{\alpha}}{\alpha}V \left(\begin{smallmatrix}0 & I_{\frac{n}{2}} \\ -\alpha I_{\frac{n}{2}} & 0\end{smallmatrix}\right) V^{-1}$ where $V^TMV = \left(\begin{smallmatrix}U_1 & 0 \\ 0 & \frac{1}{\alpha}U_1\end{smallmatrix}\right) = U^TMU$.

Now, let $Q = UV^{-1}$. We will show that $Q^{-1}AQ = B$ and $Q \in \oo(n,k,\beta)$. This will prove that $\Inn_A$ and $\Inn_B$ are isomorphic over $\oo(n,k,\beta)$ by Lemma \ref{TidyLem}. 

We first show that $Q \in \oo(n,k,\beta)$. 
\begin{align*}
Q^TMQ &= (UV^{-1})^TMUV^{-1}\\
&= (V^{-1})^T(U^TMU)V^{-1}\\&
= (V^{-1})^T(V^TMV)V^{-1}\\ 
&= M.
\end{align*}

Lastly, we show that $Q^{-1}AQ = B$. 
\begin{align*}
Q^{-1}AQ &= (UV^{-1})^{-1}A(UV^{-1})\\ 
&= VU^{-1}AUV^{-1}\\
&=-\frac{\sqrt{\alpha}}{\alpha}V \left(\begin{array}{cc}0 & I_{\frac{n}{2}} \\ -\alpha I_{\frac{n}{2}} & 0\end{array}\right) V^{-1}\\
&= B.
\end{align*}

\end{proof}

Combining the results from this section, we get the following corollary. 

\begin{cor}
\label{CorType4So}
If $\Inn_A$ and $\Inn_B$ are both Type 4 $k$-involutions of $\So(n,k,\beta)$ where $A, B \in \oo(n,k[\sqrt{\alpha}], \beta)$, then $\Inn_A$ and $\Inn_B$ are isomorphic over $\oo(n,k,\beta)$. That is, $\So(n,k,\beta)$ has at most $|k^*/(k^*)^2|-1$ isomorphy classes of Type 4 $k$-involutions.
\end{cor}

\section{Maximal Number of Isomorphy classes}

From the work we have done, it follows that the maximum number of isomorphy classes of $k$-involutions of $\So(n,k,\beta)$ over $\oo(n,k,\beta)$ is a function of the number of square classes of $k$, and the number of congruency classes of invertible diagonal matrices over $k$. We first define the following formulas.

\begin{definit}

Let $\tau_1(k) = |k^*/(k^*)^2|-1$ and $\tau_2(m,k)$ be the number of congruency classes of invertible symmetric matrices of $\Gl(m,k)$ over $\Gl(m,k)$.

Let $C_1(n,k,\beta)$, $C_2(n,k,\beta)$, $C_3(n,k,\beta)$ and $C_4(n,k,\beta)$ be the number of isomorphy classes of $\So(n,k,\beta)$ $k$-involutions over $\oo(n,k,\beta)$ of types 1, 2, 3, and 4, respectively.

\end{definit}

From our previous work, we have the following:

\begin{cor}

\begin{enumerate}

\item If $n$ is odd, then $$C_1(n,k,\beta) \le \left( \sum_{m=1}^{\frac{n-1}{2}} \tau_2(n-m,k)\tau_2(m,k) \right).$$ If $n$ is even, then 

\begin{align*}
C_1(n,k,\beta) &\le \left( \sum_{m=1}^{\frac{n}{2}-1} \tau_2(n-m,k)\tau_2(m,k) \right)+\left(\begin{array}{c}\tau_2(\frac{n}{2},k) \\2\end{array}\right)+\tau_2 \left(\frac{n}{2},k\right)\\
&= \left( \sum_{m=1}^{\frac{n}{2}-1} \tau_2(n-m,k)\tau_2(m,k) \right)+\frac{\tau(\frac{n}{2},k)(\tau(\frac{n}{2},k)+1)}{2}.\\
\end{align*}

\item  If $n$ is even, then 

\begin{align*}
C_2(n,k,\beta) &\le \tau_1(k)\left(\left(\begin{array}{c}\tau_2(\frac{n}{2},k) \\2\end{array}\right)+\tau_2 \left(\frac{n}{2},k\right)\right)\\ 
&=\tau_1(k)\left( \frac{\tau(\frac{n}{2},k)(\tau(\frac{n}{2},k)+1)}{2}\right) .\\
\end{align*}

\item  If $n$ is even, then $$C_3(n,k,\beta) \le 1 .$$ 

\item  If $n$ is even, then $$C_4(n,k,\beta) \le \tau_1(k) .$$ 

\item If $n$ is odd, then $C_2(n,k,\beta) = C_3(n,k,\beta) = C_4(n,k,\beta) = 0$.

\end{enumerate}
\end{cor}

We now list values of $\tau_1$ and $\tau_2$ for a few classes of fields.

\begin{table}[h] 
\centering
\caption { Some values of $\tau_1(k)$  }  \label{tau1}
\begin{tabular}[t]{|c||c|c|c|c|c|}
\hline  k  & $\overline{k}$ &$\mathbb{R}$ &$ \mathbb{F}_q$, $2 \not | q$ &$ \mathbb{Q}_p$, $p \ne 2$  & $ \mathbb{Q}_2$   \\
\hline $\tau_1(k)$ & 0 & 1 & 1 & 3 & 7  \\
\hline
\end{tabular}
\end{table}

\begin{table}[h] 
\centering
\caption { Some values of $\tau_2(m,k)$  }  \label{tau2}
\begin{tabular}[t]{|c||c|c|c|c|c|}
\hline  k  & $\overline{k}$ &$\mathbb{R}$ &$ \mathbb{F}_q$, $2 \not | q$    \\
\hline $\tau_2(m,k)$ & 1 & m+1 & $2$   \\
\hline
\end{tabular}
\end{table}

For the $\mathbb{Q}_p$, $\tau_2$ is a bit more difficult. Here we have 
$$\tau_2(m,\mathbb{Q}_p) = \left\{\begin{array}{c}1+\cdots \left(\begin{array}{c}3 \\m\end{array}\right),\hspace{.4 cm} m \le 3 \\2^3, \hspace{2.1 cm} m \ge 3\end{array}\right.$$

when $p \ne 2$ and 
$$\tau_2(m,\mathbb{Q}_2) = \left\{\begin{array}{c}1+\cdots \left(\begin{array}{c}7 \\m\end{array}\right),\hspace{.4 cm} m \le 7 \\2^7, \hspace{2.1 cm} m \ge 7\end{array}\right. .$$

Based on these values of $\tau_1$ and $\tau_2$, it is a straightforward matter to compute the maximal value of $C_j(n,k,\beta)$ for the fields mentioned above. We do so explicitly for the fields $\overline{k}$, $\mathbb{R}$, and $\mathbb{F}_q$ where $2 \not | q$.

\begin{cor}
\label{MaxIsomClassesSo}
Suppose $k = \overline{k}$

\begin{enumerate}
\item If $n$ is odd, then $C_1(n,\overline{k},\beta) \le \frac{n-1}{2}.$ If $n$ is even, then $C_1(n,\overline{k},\beta) \le \frac{n}{2}.$

\item  $C_2(n,\overline{k},\beta) =0.$

\item  If $n$ is odd, then $C_3(n,\overline{k},\beta) = 0$.  If $n$ is even, then $C_3(n,\overline{k},\beta) \le 1. $ 

\item  $C_4(n,\overline{k},\beta) =0.$
\end{enumerate}

Now suppose $k = \mathbb{R}$

\begin{enumerate}
\item If $n$ is odd, then 

\begin{align*}
C_1(n, \mathbb{R},\beta) &\le \sum_{m=1}^{\frac{n-1}{2}} (m+1)(n-m+1)\\
&= \frac{1}{12}(n^3+6n^2-n-6).\\
\end{align*}

If $n$ is even, then 

\begin{align*}
C_1(n, \mathbb{R},\beta) &\le \left(\sum_{m=1}^{\frac{n}{2}-1} (m+1)(n-m+1)\right) + \left(\begin{array}{c}\frac{n}{2}+1 \\2\end{array}\right) + \frac{n}{2}+1\\
&= \frac{1}{12}(n^3+6n^2+2n).\\
\end{align*}

\item  If $n$ is odd, then $C_2(n, \mathbb{R},\beta) =0.$  If $n$ is even, then 

\begin{align*}
C_2(n, \mathbb{R},\beta) &\le \left(\begin{array}{c}\frac{n}{2}+1 \\2\end{array}\right) + \frac{n}{2}+1\\
&= \frac{1}{8}(n^2+6n+8).\\
\end{align*}

\item  If $n$ is odd, then $C_3(n, \mathbb{R},\beta) = 0$.  If $n$ is even, then $C_3(n, \mathbb{R},\beta) \le 1. $ 

\item  If $n$ is odd, then $C_4(n, \mathbb{R},\beta) = 0$.  If $n$ is even, then $C_4(n, \mathbb{R},\beta) \le 1. $ 
\end{enumerate}

Lastly, suppose $k = \mathbb{F}_q$ such that $2 \not | q$.
\begin{enumerate}
\item If $n$ is odd, then $C_1(n,\mathbb{F}_q,\beta) \le 2n-6.$ If $n$ is even, then $C_1(n,\mathbb{F}_q,\beta) \le2n-1.$

\item  If $n$ is odd, then $C_2(n,\mathbb{F}_q,\beta) =0.$  If $n$ is even, then $C_2(n, \mathbb{F}_q,\beta) \le 3.$

\item  If $n$ is odd, then $C_3(n,\mathbb{F}_q,\beta) = 0$.  If $n$ is even, then $C_3(n,\mathbb{F}_q,\beta) \le 1. $ 

\item  If $n$ is odd, then $C_4(n, \mathbb{F}_q,\beta) = 0$.  If $n$ is even, then $C_4(n, \mathbb{F}_q,\beta) \le 1. $ 
\end{enumerate}

\end{cor}

\section{Explicit Examples}

\subsection{Algebraically Closed Fields}

We now find the exact number of isomorphy classes for some $\So(n,k, \beta)$. We begin by looking at the case where $k = \overline{k}$. Note that all symmetric non degenerate bilinear forms are congruent to the dot product over an algebraically closed field.

\begin{cor}

Assume $k = \overline{k}$. If $\theta$ is an $k$-involution of $\So(n,k)$, then $\theta$ is isomorphic to $\Inn_A$ where $A = \left(\begin{smallmatrix}-I_{m} & 0 \\0 & I_{n-m}\end{smallmatrix}\right)$ and $0 \le m < \frac{n}{2},$ or $A = \left(\begin{smallmatrix}0 & -I_{\frac{n}{2}} \\I_{\frac{n}{2}} & 0\end{smallmatrix}\right)$.

\end{cor}

\begin{proof}

Since $k$ is algebraically closed, we know that all $k$-involutions of $\So(n,k)$ are of Type 1 or 3. We first consider the Type 1 case. We will now find a representative matrix $A$ for each isomorphy class of Type 1 $k$-involutions. Suppose $\theta$ is a Type 1 $k$-involution. We will find a representative matrix $A$ for the isomorphic class containing $\theta$. We know we can assume $A \in \oo(n,k)$. Further, by Lemma \ref{Type1ClassSo} we can write $A = X \left(\begin{smallmatrix}-I_{m} & 0 \\0 & I_{n-m}\end{smallmatrix}\right) X^{-1}$, where we know $X^TX$ is diagonal. Since $k = \overline{k}$, then we also assume that $X^TX$ must be congruent to $I_n$. Since we are looking for a representative $A$ of our isomorphy class, we may assume $X^TX = I_n$, and we can choose $X = I_n$. This means  $A =  \left(\begin{smallmatrix}-I_{m} & 0 \\0 & I_{n-m}\end{smallmatrix}\right)$ is a representative of our isomorphy class.

We see that Type 3 $k$-involutions will exist since $J = \left(\begin{smallmatrix}0 & I_{\frac{n}{2}} \\-I_{\frac{n}{2}} & 0\end{smallmatrix}\right)$ will induce a Type 3 $k$-involution. Thus, there is one isomorphy class of Type 3 $k$-involutions.
\end{proof}

We note that in this case, that the maximal number of isomorphy classes do in fact exist. That is, in Corollary \ref{MaxIsomClassesSo}, for the case where $k = \overline{k}$, we have equality in every statement.

\subsection{The Standard Real Orthogonal Group}

We now examine the case where $\beta$ is the standard dot product, and $k = \mathbb{R}$.

\begin{cor}
If $\theta$ is an $k$-involution of $\So(n,\mathbb{R})$, then $\theta$ is isomorphic to $\Inn_A$ where $A = \left(\begin{smallmatrix}-I_{m} & 0 \\0 & I_{n-m}\end{smallmatrix}\right)$ and $0 \le m \le \frac{n}{2},$ or $A = \left(\begin{smallmatrix}0 & -I_{\frac{n}{2}} \\I_{\frac{n}{2}} & 0\end{smallmatrix}\right)$. There are no Type 2 or Type 4 $k$-involutions for this group.
\end{cor}

\begin{proof}

We first consider the Type 1 case. We will find a representative matrix $A$ for each isomorphy class of Type 1 $k$-involutions. Suppose $\theta$ is a Type 1 $k$-involution. We will find a representative matrix $A$ for the isomorphy class containing $\theta$. We know we can assume $A \in \oo(n,k)$. Further, by Lemma \ref{Type1ClassSo} we can write $A = X \left(\begin{smallmatrix}-I_{m} & 0 \\0 & I_{n-m}\end{smallmatrix}\right) X^{-1}$, where we know $X^TX$ is congruent to a diagonal where the diagonal entries are all 1's and -1's. Since we are looking for a representative of our isomorphy class, let us assume we have $X^TX$ is equal to this diagonal matrix. We see that there can be no $-1$'s in the diagonal matrix since $k = \mathbb{R}$. So, we assume $X^TX = I_n$, which means we can choose $X = I_n$. So, $A =  \left(\begin{smallmatrix}-I_{m} & 0 \\0 & I_{n-m}\end{smallmatrix}\right)$ is a representative of our isomorphy class. 

We proceed by contradiction to show that there are now Type 2 $k$-involutions of $\So(n,\mathbb{R})$. Suppose $\theta$ is a Type 2 $k$-involution. We want to find $A$ such that $\theta = \Inn_A$, By Lemma \ref{type2lemSo} we can write $A = -\frac{\sqrt{\alpha}}{\alpha} X  \left(\begin{smallmatrix}0 & I_{\frac{n}{2}} \\ \alpha I_{\frac{n}{2}} & 0\end{smallmatrix}\right)  X^{-1}$ where $X^TX = \left(\begin{smallmatrix}X_1 & 0 \\ 0 & \frac{1}{\alpha}X_1\end{smallmatrix}\right)$ is diagonal. We recall that $\alpha \in \mathbb{R}^*$ but $\sqrt{\alpha} \not \in \mathbb{R}^*.$ So, $\alpha$ must be a negative number, and we can choose $\alpha = -1$. That is,  $X^TX = \left(\begin{smallmatrix}X_1 & 0 \\ 0 & -X_1\end{smallmatrix}\right)$. But, this is a contradiction, because when $k= \mathbb{R}$, there does not exist any nonzero vectors $x$ such that $x^Tx \le 0$, so the whole diagonal of $X^TX$ must be positive, which is not possible. This shows that there are no Type 2 $k$-involutions in this case. In a similar way, we can show that there are also no Type 4 $k$-involutions in this case.

We know that there is at most one isomorphy class of Type 3 $k$-involutions by Corollary \ref{MaxIsomClassesSo}. Since $A = \left(\begin{smallmatrix}0 & -I_{\frac{n}{2}} \\ I_{\frac{n}{2}} & 0\end{smallmatrix}\right) $induces a Type 3 $k$-involution, then $A$ is a representative of the only Type 3 isomorphy class.
\end{proof}

Unlike the algebraically closed case, we note that in this case, that the maximal number of isomorphy classes do not exist. That is, in Corollary \ref{MaxIsomClassesSo}, for the case where $k = \mathbb{R}$, we have an explicit example where we do not have equality. In fact, given that we have seen that the Type 1 and 3 cases must exist for this group, we actually have the minimal number of isomorphy classes possible.

\subsection{Orthogonal Groups of $\mathbb{F}_q$}

We begin by examining the Type 1 $k$-involutions where $k = \mathbb{F}_q$ and $q = p^h$ for all cases where $p \ge 3$. This is a complete classification of the $k$-involutions when $n$ is odd. We note that for these fields we have $|(k^*)^2| = 2$. So, we will use 1 and $\delta_q$ as representatives of of the distinct field square classes. Based on properties of symmetric matrices over $k = \mathbb{F}_q$, we know that up to congruence, there are two possibilities for $M$: either $M = I_n$ or $M = \left(\begin{smallmatrix}I_{n-1} & 0 \\0 & \delta_q\end{smallmatrix}\right)$.  

\begin{theorem}

Assume that $M = I_n$. Suppose $\theta$ is a Type 1 $k$-involution of $\So(n,\mathbb{F}_q)$. Then $\theta$ is isomorphic to $\Inn_A$ where we can write $A = I_{n-m,m}$ for $0 \le m \le \frac{n}{2}$ or $$A = \left(\begin{array}{cccc}-I_{m-1} & 0 & 0 & 0  \\0 & 1-2\frac{a^2}{\delta_q} & 0 & \frac{2ab}{\delta_q}  \\0 & 0 & I_{n-m-1} & 0  \\0 & \frac{2ab}{\delta_q} & 0 & 1-2\frac{b^2}{\delta_q} \end{array}\right)$$ for $0 \le m \le \frac{n}{2}$, where $\delta_q$ is a nontrivial non-square in $\mathbb{F}_q$ where $a^2+b^2 = \delta_q$ and $a,b \in \mathbb{F}_q$.

Now, assume that $M = \left(\begin{smallmatrix}I_{n-1} & 0 \\0 & \delta_q\end{smallmatrix}\right)$. Suppose $\theta$ is an $k$-involution of $\So(n,\mathbb{F}_q, \beta)$. Then $\theta$ is isomorphic to $\Inn_A$ where we can write $$A = I_{n-m,m} \text{  \hspace{.2cm}  or  \hspace{.2cm}  } A = \left(\begin{array}{cccc}-I_{n-m-1} &   &    \\  & I_{m} &     \\  &   & -1   \end{array}\right)$$ for $0 \le m \le \frac{n}{2}$.

\end{theorem}

\begin{proof}

We will use the equivalent conditions of Lemma \ref{type1lemSo} to prove that the matrices listed above will distinctly be representatives of the isomorphy classes of the $k$-involutions of $\So(n,\mathbb{F}_q)$. For future reference, fix $a, b \in \mathbb{F}_q$ such that $a^2+b^2 = \delta_q$

If $\theta$ is a Type 1 $k$-involution, then by Lemma \ref{Type1ClassSo} we can choose a matrix $A$ such that $\theta = \Inn_A$ and we can write $A = X \left(\begin{smallmatrix}-I_s & 0 \\0 & I_t\end{smallmatrix}\right) X^{-1},$ where $s+t = n$ and $$X^TMX = \left(\begin{array}{cc}X_1 & 0 \\0 & X_2\end{array}\right)$$ must be diagonal, and $X_1$ is an $s \times s$ matrix, and $X_2$ is a $t \times t$ matrix. It is a well known fact that any diagonal matrix over $\mathbb{F}_q$ must be congruent to either $I_n$ or $\left(\begin{smallmatrix}I_{n-1} & 0 \\0 & \delta_q\end{smallmatrix}\right)$ where $\delta_q$ is some fixed non-square in $\mathbb{F}_q$. So, from the equivalent conditions in Theorem \ref{type1lemSo} it is known that $X_1$ and $X_2$ must each be congruent to $I$ or $\left(\begin{smallmatrix}I & 0 \\0 & \delta_q\end{smallmatrix}\right)$ (sizing the matrices appropriately). Further, since $\det(X^TX) = (\det(X))^2$ is a square, we observe that $X_1$ and $X_2$ must be simultaneously congruent to either $I$ or $\left(\begin{smallmatrix}I & 0 \\0 & \delta_q\end{smallmatrix}\right)$ (again, sizing appropriately). 

 Since we are searching for a representative of the congruence class, it can be assumed that  $X^TMX$ is either $I$ or $\left(\begin{smallmatrix}I & 0 & 0 & 0 \\0 & \delta_q & 0 & 0 \\0 & 0 & I & 0 \\0 & 0 & 0 & \delta_q\end{smallmatrix}\right)$. These are the only possibilities, and also they must correspond to distinct isomorphy classes of Type 1 $k$-involutions under the conditions of Theorem \ref{type1lemSo}. 
 
{\bf Case 1}: $\beta$ is the standard dot product, and $M = I$. 

{\bf Subcase 1.1}: $X^TX = I$.

We can let $X = I$, which means $A =  \left(\begin{smallmatrix}-I_s & 0 \\0 & I_t\end{smallmatrix}\right)$ is the representative of the isomorphy class.

{\bf Subcase 1.2}: $X^TX = \left(\begin{smallmatrix}I & 0 & 0 & 0 \\0 & \delta_q & 0 & 0 \\0 & 0 & I & 0 \\0 & 0 & 0 & \delta_q\end{smallmatrix}\right)$.
 
We can let $$X = \left(\begin{array}{cccc}I & 0 & 0 & 0 \\0 & a & 0 & b \\0 & 0 & I & 0 \\0 & -b & 0 & a\end{array}\right).$$ It follows from this that $$A = \left(\begin{array}{cccc}-I_{m-1} & 0 & 0 & 0  \\0 & 1-2\frac{a^2}{\delta_q} & 0 & \frac{2ab}{\delta_q}  \\0 & 0 & I_{n-m-1} & 0  \\0 & \frac{2ab}{\delta_q} & 0 & 1-2\frac{b^2}{\delta_q} \end{array}\right)$$ is a representative of the isomorphy class.
 
 {\bf Case 2}: $\beta$ is such that $M = \left(\begin{smallmatrix}I_{n-1} & 0 \\0 & \delta_q\end{smallmatrix}\right).$ 
 
 {\bf Subcase 2.1}: $X^TMX = M$.
 
In the first case, since we are looking for a representative of our congruence class, we can assume $X^T\left(\begin{smallmatrix}I_{n-1} & 0 \\0 & \delta_q\end{smallmatrix}\right)X = \left(\begin{smallmatrix}I_{n-1} & 0 \\0 & \delta_q\end{smallmatrix}\right)$. This means we can assume $X=I$ choose $A =  \left(\begin{smallmatrix}-I_s & 0 \\0 & I_t\end{smallmatrix}\right)$ as the representative of the isomorphy class.

{\bf Subcase 2.2}: $X^TMX =   \left(\begin{smallmatrix}I_{\frac{n}{2}-1} & 0 & 0 \\0 & \delta_q & 0 \\0 & 0 & I_{\frac{n}{2}}\end{smallmatrix}\right).$

We can choose $X = \left(\begin{smallmatrix}I_{s-1} &   &   &   \\  & 0 &   & 1 \\  &   & I_{t-1} &   \\  & 1 &   & 0\end{smallmatrix}\right)$. This gives representative $A = \left(\begin{smallmatrix}-I_{s-1} &   &    \\  & I_t &     \\  &   & -1   \end{smallmatrix}\right).$
\end{proof}

By counting the number of isomorphy classes  from this Theorem, its clear that if $n$ is odd, then $C_1(n,\mathbb{F}_q,\beta) = n+1,$ and if $n$ is even, then $C_1(n,\mathbb{F}_q,\beta) = n+2.$ 

For the remaining three types of $k$-involutions, we restrict our attention to the case where $\beta$ is the standard dot product. Recall that $n$ must be even. In this case, it is clear that $A = \left(\begin{array}{cc}0 & I_n \\-I_n & 0\end{array}\right)$ will induce a Type 3 $k$-involution, and that $C_3(n, k)= 1$. 

We know that $C_2(n,\mathbb{F}_q) \le 3$ and $C_4(n,\mathbb{F}_q) \le 1$. We will specifically look at the cases where $q = $3, 5, and 7. For these cases, we see that we have existence of both Type 2 and Type 4 $k$-involutions via the matrices in Table \ref{So_Fp}. 

\begin{table}[h] 
\centering
\caption { Type 2 and Type 4 examples for $\So(4,\mathbb{F}_p$) }  \label{So_Fp}
\begin{tabular}[t]{|c||c|c|}
\hline  $k$  & Type 2 & Type 4     \\
\hline $\mathbb{F}_3$ &$ i \left(\begin{smallmatrix}1 & 1 & 0 & 0 \\1 & 2 & 0 & 0 \\0 & 0 & 1 & 1 \\0 & 0 & 1 & 2\end{smallmatrix}\right)$  &  $i\left(\begin{smallmatrix}1 & 2 & 0 & 0 \\1 & 1 & 0 & 0 \\0 & 0 & 1 & 2 \\0 & 0 & 1 & 1\end{smallmatrix}\right) $  \\
\hline $\mathbb{F}_5$ & $\sqrt{2}\left(\begin{smallmatrix}1 & 1 & 0 & 0 \\1 & 4 & 0 & 0 \\0 & 0 & 1 & 1 \\0 & 0 & 1 & 4\end{smallmatrix}\right) $ &  $\sqrt{2}\left(\begin{smallmatrix}1 & 4 & 0 & 0 \\1 & 1 & 0 & 0 \\0 & 0 & 1 & 4 \\0 & 0 & 1 & 1\end{smallmatrix}\right) $  \\
\hline $\mathbb{F}_7$ & $\sqrt{3}\left(\begin{smallmatrix}1 & 3 & 0 & 0 \\3 & 6 & 0 & 0 \\0 & 0 & 1 & 3 \\0 & 0 & 3 & 6\end{smallmatrix}\right) $ &  $\sqrt{3}\left(\begin{smallmatrix}1 & 4 & 0 & 0 \\3 & 1 & 0 & 0 \\0 & 0 & 1 & 4 \\0 & 0 & 3 & 1\end{smallmatrix}\right) $  \\
\hline
\end{tabular}
\end{table}

We note that these examples will all generalize to higher dimensions, so it is clear that for these fields that whenever $n$ is even,  $C_2(n,\mathbb{F}_q), C_4(n,\mathbb{F}_q) \ge 1$. So, for these three specific fields, we know that $C_4(n,\mathbb{F}_q) =1$, and that the number of isomorphy classes of Type 4 $k$-involutions are maximized. But, for $\So(4, \mathbb{F}_p)$ where $p =$ 3, 5, and 7, we have done computations in Maple which use the conditions of Theorem \ref{type2lemSo} that show that$C_2(4, \mathbb{F}_p) = 1$. So, the number of Type 2 isomorphy classes is not maximized in these cases. While we have been unable to prove this up to this point, we believe that this is a pattern that would continue. That is, we have the following conjecture:

\begin{conj}
Suppose that $\So(n,k)$ is a finite orthogonal group and that $n$ is even. Then,  $C_2(n,k)=C_4(n,k) = 1$.
\end{conj}

\subsection{$p$-adic numbers}

We now turn our attention to the case where $k= \mathbb{Q}_p$. We will assume $M = I_n$. We show a classification of the possible isomorphy classes of the Type 1 $k$-involutions of $\So(n, \mathbb{Q}_p)$ where $p >2$, using Lemma \ref{type1lemSo}. Note that if $n$ is odd and $n \ne 3$, then this all of the possible isomorphy classes of the $k$-involutions of $\So(n, \mathbb{Q}_p)$. We note that we say ``possible" because we don't show existence, but rather we use our characterization of Type 1 $k$-involutions to show which classes may exist. It still remains to be shown which of these possible classes does exist.

We first state a result from  \cite{Jones} about symmetric matrices with entries from the p-adic numbers. 

\begin{lem} 
\label{padicSo} 
Symmetric matrices $M_1$ and $M_2$ with entries in $\mathbb{Q}\sb{p}$ are congruent if and only if 
$$\det(M_1)=\gamma^{2}\det(M_2)
\hspace{1 em} and \hspace{1 em} c\sb{p}(M_1)=c\sb{p}(M_2) $$
where $c_p(M)$ denotes the Hasse symbol of matrix $M$.
\end{lem}

We use this to prove a result that is an extension of Theorem \ref{type1lemSo} in the case that $k= \mathbb{Q}_p$.

\begin{cor}
\label{padicCor}
Assume the hypotheses of Theorem \ref{type1lemSo}. Statements $(i)$ through $(v)$ of Theorem \ref{type1lemSo} are equivalent the following condition:

There exists some $\gamma \in \mathbb{Q}_p$ such that 
$$\det(X_1) =\gamma^2\det(Y_1),   \hspace{.2cm}  \det(X_2) =\gamma^2\det(Y_2),   \hspace{.2cm} c_p(X_1)=c_p(Y_1),   \hspace{.2cm} \& \hspace{.2cm} c_p(X_2)=c_p(Y_2)$$ or 
$$\det(X_1) =\gamma^2\det(Y_2),   \hspace{.2cm}  \det(X_2) =\gamma^2\det(Y_1),   \hspace{.2cm} c_p(X_1)=c_p(Y_2),   \hspace{.2cm} \& \hspace{.2cm} c_p(X_2)=c_p(Y_1).$$
\end{cor}

\begin{proof}
We note that this condition is equivalent to $(iii)$ of Theorem \ref{type1lemSo} by Lemma \ref{padicSo}. 
\end{proof}

Corollary \ref{padicCor} gave us conditions on the square class of the determinant and the Hasse symbol to classify the isomorphy classes for $\So(n, \mathbb{Q}_p)$. Using these conditions, we have classified all of the possible isomorphy classes of Type 1 $k$-involutions based on what the values of $X_1$ and $X_2$ would be for a representative of the congruency class in Tables \ref{so-Qp-table1} and \ref{so-Qp-table2}. We note that each isomorphy class is determined by the triple $(\det(X_1) = \det(X_2), c_p(X_1),c_p(X_2)).$ To show that each of these possible congruency classes exists, one would need to find a matrix $X$ such that $X^TX = \left(\begin{array}{cc}X_1 & 0 \\0 & X_2\end{array}\right).$ This would then determine $A$. In the case where $-1 \not \in (\mathbb{Q}_p^*)^2$, this will always be the case. To see that this is true, note that $X^TX = \left(\begin{array}{cc}X_1 & 0 \\0 & X_2\end{array}\right)$ will always be a symmetric matrix with a determinant that is in the same square class as 1. When $-1 \not \in(\mathbb{Q}_p^*)^2$, all such matrices are such that $c_p\left(\begin{array}{cc}X_1 & 0 \\0 & X_2\end{array}\right) = 1$ is the case. So, $\left(\begin{array}{cc}X_1 & 0 \\0 & X_2\end{array}\right)$ will be congruent to $I_n$, which gives us the existence of $X$ such that $X^TX = \left(\begin{array}{cc}X_1 & 0 \\0 & X_2\end{array}\right)$. In the case where $-1 \in (\mathbb{Q}_p^*)^2$, then it is possible that $c_p\left(\begin{array}{cc}X_1 & 0 \\0 & X_2\end{array}\right) = -1$. For these cases, it is not clear (to the authors) that there exists $X$ such that $X^TX = \left(\begin{array}{cc}X_1 & 0 \\0 & X_2\end{array}\right)$.

\begin{table}[h] 
\centering
\caption {$X_1$ and $X_2$ values when $k= \mathbb{Q}_p$, $p>2$, and $-1\in (\mathbb{Q}_p^*)^2$}  \label{so-Qp-table1}
\begin{tabular}[t]{|c|c|c|c|c|c|c|c|}
\hline  $X_1$ &$X_2$ & $\det(X_1)$ and $\det(X_2)$ &$c_p(X_1)$ &$c_p(X_2)$ \\
\hline $I_n$ & $I_n$ & 1 & 1& 1 \\
\hline $I_n$ & \tiny{$ \left(\begin{array}{cccc}I_{n-3} & 0 & 0 & 0 \\0 & p & 0 & 0 \\0 & 0 & N_p & 0 \\0 & 0 & 0 & pN_p\end{array}\right)$} & 1 &  1& -1 \\
\hline \tiny{$ \left(\begin{array}{cc}I_{n-1} & 0 \\0 & p\end{array}\right)$} & \tiny{$  \left(\begin{array}{cc}I_{n-1} & 0 \\0 & p\end{array}\right)$} & $p$ &  1& 1 \\
\hline \tiny{$ \left(\begin{array}{cc}I_{n-1} & 0 \\0 & p\end{array}\right)$} & \tiny{$ \left(\begin{array}{ccc}I_{n-2} & 0 & 0 \\0 & N_p & 0 \\0 & 0 & pN_p\end{array}\right)$} & $p$ & 1& -1 \\
\hline \tiny{$ \left(\begin{array}{cc}I_{n-1} & 0 \\0 & N_p\end{array}\right)$} & \tiny{$  \left(\begin{array}{cc}I_{n-1} & 0 \\0 & N_p\end{array}\right)$} & $N_p$ &  1& 1 \\
\hline \tiny{$ \left(\begin{array}{cc}I_{n-1} & 0 \\0 & N_p\end{array}\right)$} & \tiny{$ \left(\begin{array}{ccc}I_{n-2} & 0 & 0 \\0 & p & 0 \\0 & 0 & pN_p\end{array}\right)$}  & $N_p$& 1& -1 \\
\hline \tiny{$ \left(\begin{array}{cc}I_{n-1} & 0 \\0 & pN_p\end{array}\right)$} & \tiny{$  \left(\begin{array}{cc}I_{n-1} & 0 \\0 & pN_p\end{array}\right)$} & $pN_p$ & 1& 1 \\
\hline \tiny{$ \left(\begin{array}{cc}I_{n-1} & 0 \\0 & pN_p\end{array}\right)$} & \tiny{$ \left(\begin{array}{ccc}I_{n-2} & 0 & 0 \\0 & p & 0 \\0 & 0 & N_p\end{array}\right)$} &  $pN_p$& 1& -1 \\
\hline \tiny{$ \left(\begin{array}{ccc}I_{n-2} & 0 & 0 \\0 & p & 0 \\0 & 0 & N_p\end{array}\right)$} & \tiny{$ \left(\begin{array}{ccc}I_{n-2} & 0 & 0 \\0 & p & 0 \\0 & 0 & N_p\end{array}\right)$}  & $pN_p$& -1& -1 \\
\hline \tiny{$ \left(\begin{array}{ccc}I_{n-2} & 0 & 0 \\0 & p & 0 \\0 & 0 & pN_p\end{array}\right)$} & \tiny{$ \left(\begin{array}{ccc}I_{n-2} & 0 & 0 \\0 & p & 0 \\0 & 0 & pN_p\end{array}\right)$}  & $N_p$& -1& -1 \\
\hline \tiny{$ \left(\begin{array}{ccc}I_{n-2} & 0 & 0 \\0 & N_p & 0 \\0 & 0 & pN_p\end{array}\right)$} & \tiny{$ \left(\begin{array}{ccc}I_{n-2} & 0 & 0 \\0 & N_p & 0 \\0 & 0 & pN_p\end{array}\right)$} &  $p$& -1& -1 \\
\hline \tiny{$  \left(\begin{array}{cccc}I_{n-3} & 0 & 0 & 0 \\0 & p & 0 & 0 \\0 & 0 & N_p & 0 \\0 & 0 & 0 & pN_p\end{array}\right)$} & \tiny{$\left(\begin{array}{cccc}I_{n-3} & 0 & 0 & 0 \\0 & p & 0 & 0 \\0 & 0 & N_p & 0 \\0 & 0 & 0 & pN_p\end{array}\right)$ }& $1$ &  -1& -1 \\
\hline
\end{tabular}
\end{table}

\begin{table}[h] 
\centering
\caption {$X_1$ and $X_2$ values when $k= \mathbb{Q}_p$, $p>2$ and $-1\not \in(\mathbb{Q}_p^*)^2$}  \label{so-Qp-table2}
\begin{tabular}[t]{|c|c|c|c|c|c|c|c|}
\hline  $X_1$ &$X_2$ & $\det(X_1)$ and $\det(X_2)$ &$c_p(X_1)$ &$c_p(X_2)$ \\
\hline $I_n$ & $I_n$ & 1 & 1& 1 \\
\hline $ \left(\begin{array}{cc}I_{n-1} & 0 \\0 & p\end{array}\right)$ & $  \left(\begin{array}{cc}I_{n-1} & 0 \\0 & p\end{array}\right)$ & $p$ &  -1& -1 \\
\hline $ \left(\begin{array}{cc}I_{n-1} & 0 \\0 & p\end{array}\right)$ & $ \left(\begin{array}{ccc}I_{n-2} & 0 & 0 \\0 & N_p & 0 \\0 & 0 & pN_p\end{array}\right)$ & $p$ &- 1& 1 \\
\hline $ \left(\begin{array}{cc}I_{n-1} & 0 \\0 & N_p\end{array}\right)$ & $  \left(\begin{array}{cc}I_{n-1} & 0 \\0 & N_p\end{array}\right)$ & $N_p$ &  1& 1 \\
\hline $ \left(\begin{array}{cc}I_{n-1} & 0 \\0 & pN_p\end{array}\right)$ & $  \left(\begin{array}{cc}I_{n-1} & 0 \\0 & pN_p\end{array}\right)$ & $pN_p$ & -1& -1 \\
\hline $ \left(\begin{array}{cc}I_{n-1} & 0 \\0 & pN_p\end{array}\right)$ & $ \left(\begin{array}{ccc}I_{n-2} & 0 & 0 \\0 & p & 0 \\0 & 0 & N_p\end{array}\right)$ &  $pN_p$& -1& 1 \\
\hline $ \left(\begin{array}{ccc}I_{n-2} & 0 & 0 \\0 & p & 0 \\0 & 0 & N_p\end{array}\right)$ & $ \left(\begin{array}{ccc}I_{n-2} & 0 & 0 \\0 & p & 0 \\0 & 0 & N_p\end{array}\right)$  & $pN_p$& 1& 1 \\
\hline $ \left(\begin{array}{ccc}I_{n-2} & 0 & 0 \\0 & N_p & 0 \\0 & 0 & pN_p\end{array}\right)$ & $ \left(\begin{array}{ccc}I_{n-2} & 0 & 0 \\0 & N_p & 0 \\0 & 0 & pN_p\end{array}\right)$ &  $p$& 1& 1 \\
\hline
\end{tabular}
\end{table}

We now assume that $p = 2$, and we construct a classification of the Type 1 $k$-involutions. We again note that if $n$ is odd and $n \ne 3$, then this is a complete classification. We see that $\pm 1$, $\pm 2$, $\pm 3$ and $\pm 6$ are representatives for all of the the distinct square classes of $(\mathbb{Q}_2^*)^2$. For this case, we have not constructed tables with complete classifications of the two sets of isomorphy classes. Instead, we have constructed a table, Table \ref{so-Qp-table5}, where there is a diagonal matrix over $\mathbb{Q}_2$ for each possible pair of determinant square class and value of Hasse symbol. A potential isomorphy class is determined by choosing for $X_1$ and $X_2$ any pair of matrices on this table where the two given matrices have determinants in the same  square class. So, given the different possible Hasse symbol values, there are at most 24 isomorphy classes of Type 1 $k$-involutions. As in some of the previous cases, it is not immediately clear that there does or does not exist a matrix $X$ in each of these cases such that $X^TX =  \left(\begin{array}{cc}X_1 & 0 \\0 & X_2\end{array}\right).$

\begin{table}[h] 
\centering
\caption { $X_1$ and $X_2$ values when $k= \mathbb{Q}_2$ }  \label{so-Qp-table5}
\begin{tabular}[t]{|c|c|c|}
\hline  $\det(Y)$ square class  & $c_f(Y) = 1$  &$c_p(Y) = -1$  \\
\hline $1$  & $ I_n$ & $ \left(\begin{array}{cccc}I_{n-3} & 0 & 0 & 0 \\0 & -2 & 0 & 0 \\0 & 0 & 3 & 0 \\0 & 0 & 0 & -6\end{array}\right)$  \\
\hline $-1$ & $\left(\begin{array}{ccc}I_{n-2} & 0 & 0 \\0 & 2 & 0 \\0 & 0 & -2\end{array}\right)$ &  $ \left(\begin{array}{cc}I_{n-1} & 0 \\0 & -1\end{array}\right)$ \\
\hline $ 2$  & $\left(\begin{array}{ccc}I_{n-2} & 0 & 0 \\0 & -1 & 0 \\0 & 0 & -2\end{array}\right)$ & $  \left(\begin{array}{cc}I_{n-1} & 0 \\0 & 2\end{array}\right)$ \\
\hline $ -2$  & $  \left(\begin{array}{cc}I_{n-1} & 0 \\0 & -2\end{array}\right)$ &  $ \left(\begin{array}{cccc}I_{n-3} & 0 & 0 & 0 \\0 & -1 & 0 & 0 \\0 & 0 & -3 & 0 \\0 & 0 & 0 & -6\end{array}\right)$ \\
\hline $ 3$  & $  \left(\begin{array}{cc}I_{n-1} & 0 \\0 & 3\end{array}\right)$ & $\left(\begin{array}{ccc}I_{n-2} & 0 & 0 \\0 & 2 & 0 \\0 & 0 & 6\end{array}\right)$ \\
\hline $ -3$  & $\left(\begin{array}{ccc}I_{n-2} & 0 & 0 \\0 & -1 & 0 \\0 & 0 & -3\end{array}\right)$ & $  \left(\begin{array}{cc}I_{n-1} & 0 \\0 & -3\end{array}\right)$ \\
\hline $ 6$  & $  \left(\begin{array}{cc}I_{n-1} & 0 \\0 & 6\end{array}\right)$ & $\left(\begin{array}{ccc}I_{n-2} & 0 & 0 \\0 & 2 & 0 \\0 & 0 & 3\end{array}\right)$ \\
\hline $ -6$  & $\left(\begin{array}{ccc}I_{n-2} & 0 & 0 \\0 & -1 & 0 \\0 & 0 & 6\end{array}\right)$ & $  \left(\begin{array}{cc}I_{n-1} & 0 \\0 & -6\end{array}\right)$ \\
\hline
\end{tabular}
\end{table}

%
%

\end{document}